\documentclass[reqno, 12pt]{amsart}

\usepackage{_pascal/_pascal}

\hypersetup{
    pdftitle={},
    pdfauthor={Maria Chudnovsky, J.~Pascal Gollin, Matjaz Krnc, Martin Milanič}
}

\newcommand{\tw}[0]{\ensuremath{\mathsf{tw}}}

\newcommand{\wedgeselectionnull}[0]{\ensuremath{\Omega}}
\newcommand{\wedgeselection}[1]{\ensuremath{\Omega_{#1}}}
\newcommand{\segmentselectionnull}[0]{\ensuremath{\sigma}}
\newcommand{\segmentselection}[1]{\ensuremath{\segmentselectionnull(#1)}}
\newcommand{\barrier}[1]{\ensuremath{\beta}(#1)}

\usepackage[normalem]{ulem}

\begin{document}

\author[Chudnovsky]{Maria Chudnovsky$^\ast$}
\address[$\ast$]{Princeton University, Princeton, NJ, USA.
Supported by NSF Grant
 DMS-2348219, NSF Grant CCF-2505100, AFOSR grant FA9550-22-1-0083 and a Guggenheim Fellowship.}
\email{\tt mchudnov@math.princeton.edu}
\author[Gollin]{J.~Pascal Gollin$^\dagger$}
\address[$\dagger$]{FAMNIT, University of Primorska, Koper, Slovenia. Supported in part by the Slovenian Research and Innovation Agency (I0-0035, research programs P1-0285 and P1-0383, and research projects J1-3003, J1-4008, J1-4084, J1-60012, and N1-0370) and by the research program CogniCom (0013103) at the University of Primorska.}
\email{\tt pascal.gollin@famnit.upr.si}

\author[Krnc]{Matja\v{z} Krnc$^{\dagger,\ddagger}$}
\email{\tt matjaz.krnc@upr.si}

\author[Milani\v{c}]{Martin Milani\v{c}$^{\dagger,\ddagger}$}
\address[$\ddagger$]{IAM, University of Primorska, Koper, Slovenia.}
\email{\tt martin.milanic@upr.si}

\title[Dominated balanced separators in wheel-induced-minor-free graphs]{Dominated balanced separators\\ in wheel-induced-minor-free graphs}




\begin{abstract}
    Gartland and Lokshtanov conjectured that every graph that excludes some planar graph as an induced minor has a \emph{balanced separator}, that is, a separator whose deletion leaves every component with no more than half of the vertices of the graph, which is dominated by a bounded number of vertices. 
    We confirm this conjecture for excluding any fixed \emph{wheel}, that is, a cycle together with a universal vertex, as an induced minor. 
\end{abstract}

\maketitle

\section{Introduction}
\label{sec:intro}

In their celebrated Grid Minor Theorem, Robertson and Seymour~\cite{RobertsonS:1986:GraphMinorsV} proved in 1986 that excluding any planar graph as a minor bounds the treewidth of a graph class. 
This result can be reformulated using the concept of balanced separators. 
We call a set~${S \subseteq V(G)}$ of vertices a \emph{balanced separator} of~$G$ (for a set~${X \subseteq V(G)}$) if no component of~${G - S}$ contains more than half of the vertices of~$G$ (or~$X$). 
Note that we do not require~$S$ to be a proper separator, that is, we do not require that~${G - S}$ has at least two components. 
Robertson and Seymour~\cite{RobertsonS:1986:GraphMinorsII,RobertsonS:1995:GraphMinorXIII} and Reed~\cite{Reed:1997:TWandTangles} also proved that the treewidth of a graph is linearly tied to the smallest integer~$k$ such that for every~${X \subseteq V(G)}$, the graph~$G$ has a balanced separator for~$X$ of size at most~$k$. 
Hence, as a consequence of the Grid Minor Theorem, for every planar graph~$H$ there exists an integer~$k_H$ such that every $H$-minor-free graph has a balanced separator of size at most~$k_H$. 
In fact, as a consequence of a result of Dvo\v{r}\'{a}k and Norin~\cite{DvorakNorin:2019:BalancedSeps}, the existence of small balanced separators in every subgraph implies bounded treewidth. 
Hence, the Grid Minor Theorem could be equivalently restated as follows: excluding any planar graph as a minor implies the existence of a small balanced separator.

\smallskip

In recent years, there has been a renewed interest to push these types of results into the setting of \emph{induced minors}, where we say that a graph~$H$ is an induced minor of a graph~$G$ if~$H$ is isomorphic to a graph that can be obtained from~$G$ by deleting vertices and contracting edges (and deleting any resulting parallel edges). 
While the setting of induced minors is much more challenging than that of minors (see, e.g.,~\cite{AboulkerBPT:2025:IndDisPaths}), graph classes closed under induced minors enjoy several good structural and algorithmic properties (see, e.g.,~\cite{GajarskyJLMPRS:2024,KorhonenL:2024:InducedMinorFree,AhnGHL:2025:CoarseEP,Korhonen:2023:inducedgridmaxdegree,BonnetDGTW:2023:MIS,BousquetDDHMPT:2026:IndMinorModels,CampbellDDFGHHWWY:2024+}).

Gartland and Lokshtanov conjectured the following variant of the Grid Minor Theorem for the induced setting, 
see the doctoral thesis of Gartland~\cite{Gartland:2023:Thesis}. 
To formulate it, we say a set~${S \subseteq V(G)}$ is \emph{dominated} by a set~${C \subseteq V(G)}$ if~${S \subseteq N[C]}$. 
Moreover, we say that for a graph~$H$, a graph~$G$ is \emph{$H${\dash}induced-minor-free} if~$G$ does not contain~$H$ as an induced minor. 

\begin{conjecture}
    \label{conjecture:main}
    For every planar graph~$H$ there exists an integer~$k_H$ such that every $H${\dash}induced-minor-free graph~$G$ has a balanced separator dominated by at most~$k_H$ vertices. 
\end{conjecture}

It is reasonable to also consider a weighted variant of balanced separators{\footnotemark}, 
as well as to relax the domination requirement by requiring instead that the separator is contained in a bounded number of balls of bounded diameter. 
In this regard, Abrishami et al.~\cite{AbrishamiCKPPR:2025} conjectured that the existence of weighted balanced separators contained in a bounded number of balls of bounded radius implies the existence of a tree-decomposition{\setcounter{footnote}{0}\footnotemark} such that every bag is contained in a bounded number of balls of bounded radius (possibly for larger bounds).
As shown by Abrishami et al., the radius~$1$ case of the conjecture is equivalent to the general case.
For radius~$1$, the assumption of the conjecture is equivalent to the existence of a weighted balanced separator dominated by a small set of vertices.
If this can be shown to imply the existence of a tree-decomposition such that every bag is contained in a bounded number of balls of radius $1$, this would imply, together with a weighted variant of \zcref{conjecture:main}, that for every planar graph~$H$ there exists an integer~$k_H$ such that every $H${\dash}induced-minor-free graph~$G$ admits a tree-decomposition in which each bag is dominated by at most~$k_H$ vertices.
Replacing in this statement the induced minor relation with the minor relation and strengthening the domination requirement with a bound on the cardinality of the bags restores the original Grid Minor Theorem. 

\footnotetext{For a precise definition, see \zcref{sec:prelims}.}

Several partial results regarding \zcref{conjecture:main} are known.
For example, the conjecture is valid if~$H$ is a path, by a classic argument due to Gy\'{a}rf\'{a}s~\cite{Gyarfas:1985:GyarfasPath}, see also~\cite{ChudnovskyPPT:2024:ApproxMWIS}. 
The conjecture is also known to hold for several other cases, including when~$H$ is a cycle~\cite{ChudnovskyPPT:2024:ApproxMWIS}, a disjoint union of cycles of length~$3$~\cite{AhnGHL:2025:CoarseEP}, a subdivided claw~\cite{ChudnovskyCLMS:2025:TIN5}, $K_5$ minus an edge, the $4$-wheel (that is, the graph obtained from a cycle of length~$4$ by adding a universal vertex), or a complete bipartite graph with exactly two vertices in one of the parts of the bipartition, see~\cite{DallardMS:2024:TWvsOmega3}. 

\smallskip
In this paper, we prove a common generalization of the aforementioned results for the cases when $H$ is a path, a cycle, or the $4$-wheel, by verifying \zcref{conjecture:main} for the case when $H$ is any wheel. 
For an integer~${\ell \geq 3}$, we denote by~$W_\ell$ the \emph{$\ell$-wheel}, that is, the graph obtained from a cycle of length~${\ell}$ by adding a universal vertex. 

\begin{theorem}
    \label{thm:main}
    For every integer~${\ell \geq 3}$ there exists an integer~$k$ such that every $W_\ell${\dash}induced-minor-free graph~$G$ has a balanced separator dominated by at most~$k$ vertices. 
\end{theorem}

In fact, we show the stronger weighted variant of this theorem, see \zcref{thm:mainWeighted}.

\paragraph{Algorithmic implications}
Given a graph $G$ with weights on its vertices, the \textsc{Maximum Weight Independent Set (MWIS)} problem is the problem of finding an independent set in $G$ of maximum total weight. \zcref{thm:main} implies that \textsc{MWIS} can be solved in subexponential time on the class of $W_\ell${\dash}induced-minor-free graphs; 
it also implies the existence of a QPTAS (quasi-polynomial-time approximation scheme) for MWIS in this graph class. 
To see that one can follow the proof of 
Theorem~5.1 in \cite{ChudnovskyPPT:2024:ApproxMWIS}, and then apply 
Theorems~3.5 and~3.6 from \cite{ChudnovskyPPT:2024:ApproxMWIS}.
Theorem~5.1 of \cite{ChudnovskyPPT:2024:ApproxMWIS}  deals with graphs with no long induced paths, but their only property used in the proof is the existence of a balanced separator with small domination number. 
For more details, see also~\cite[Theorems 1.2.5 and 1.2.6]{Gartland:2023:Thesis}, which also makes use of results from Bacs\'o et al.~\cite{BascoLMPTvL:2019}. 

\begin{corollary}
    \label{cor:mwis}
    For every integer~${\ell \geq 3}$, there exists a subexponential-time algorithm, as well as a QPTAS, for \textnormal{\textsc{MWIS}} 
    restricted to the class of $W_\ell${\dash}induced-minor-free graphs. 
\end{corollary}

In fact, there are many more problems that can be solved in subexponential time on hereditary graph classes in which every graph has a dominated balanced separator, including \textsc{List 3-Colouring} \cite{Pawel:2025:PC}, \textsc{Feedback Vertex Set}, \textsc{Maximum Induced Matching}, and more~\cite{NovotnaOPRvLW:2021}; see also~\cite{KorhonenL:2024:InducedMinorFree}. 
For algorithmic applications, balanced separators that are dominated by a set of logarithmic or polylogarithmic size are often enough (see, e.g., \cite{MajewskiMMOPRS:2024:MWISnoLongClaws}).

\paragraph{Related work}
Induced variants of the Grid Minor Theorem have been studied in restricted settings.
For example, Korhonen \cite{Korhonen:2023:inducedgridmaxdegree} proved that if, in addition to excluding a planar graph as an induced minor, we bound the maximum degree of a graph class, that is, exclude some star~$K_{1,d}$ as a subgraph, then the treewidth of the graph class is bounded. 
Observe that for graphs of bounded degree, all the bags of a tree-decomposition of a graph are of bounded size if and only if each of them is dominated by a small number of vertices. 
In particular, \zcref{conjecture:main} holds for graphs of bounded maximum degree. 

A more general situation is when a star~${K_{1,d}}$ is excluded as an induced subgraph.
In this case, the above equivalence generalizes as follows: 
each bag of a tree-decomposition is dominated by a small number of vertices if and only if each bag induces a subgraph of bounded independence number.
Again, these conditions imply the existence of a balanced separator dominated by a small number of vertices. 
Hence, in the case when a fixed star~${K_{1,d}}$ is excluded as an induced subgraph, \zcref{conjecture:main} is equivalent to a conjecture by Dallard et al.~\cite{DallardKKMMW:2024} stating that, for every class of graphs excluding~$K_{1,d}$ as an induced subgraph and every planar graph~$H$ there exists an integer~$k$ such that every $H${\dash}induced-minor-free graph in the class has tree-independence number{\footnotemark} at most~$k$.
\footnotetext{The \emph{tree-independence number} of a graph~$G$ is the smallest integer~$k$ such that~$G$ admits a tree-decomposition in which every bag induces a subgraph with independence number at most~$k$ (see, e.g., \cite{Yolov:2018:YolovWidth,DallardMS:2024:TWvsOmega2,LimaMMORS:2024}).}
Besides the aforementioned cases of $H$ for which \zcref{conjecture:main} is known to hold, the conjecture of Dallard et al.~was verified by Choi et at.~\cite{ChoiHMW:2025} for the case when $H$ is a wheel.
Observe that \zcref{thm:main} generalizes this result.

Choi and Wiederrecht~\cite{ChoiWiederrecht:2025} gave an improved bound on the resulting tree-independence number and verified the conjecture of Dallard et al.~for the case where the planar graph that is excluded as an induced minor is a \emph{$k$-skinny ladder}, that is, a graph obtained from two disjoint $k$-vertex paths~$P$ and~$Q$ by adding for each~${i \in [k]}$ a new vertex~$v_i$ adjacent to the $i$-th vertex of~$P$ and the $i$-th vertex of~$Q$. 

Another, somewhat more general, related conjecture is the \emph{coarse grid minor conjecture} from Georgakopoulos and Papasoglu~\cite{GeorgakopoulosP:2025:CoarseGMT} which is formulated in the language of metric graph theory and ``asymptotic minors''. 
This can be seen as a ``distance version'' of the grid minor theorem related to the notion of ``fat minors'', which were independently introduced by Chepoi et al.~\cite{ChepoiDNRV:2012:FatMinors} and Bonamy et al.~\cite{BonamyBEGLPS:2024:FatMinors}. 
A weaker version of this conjecture in this language was conjectured by Davies et al.~\cite{DaviesHIM:2024:FarMinorsThinning}. 
However, both of these conjectures have recently been refuted by Albrechtsen and Davies~\cite{AlbrechtsenD:2025:CoarseGMTCounterexample}. 
Their counterexample leaves open the case where fat-minors correspond to induced minors, so the conjectures of Dallard et al.~\cite{DallardKKMMW:2024} and \zcref{conjecture:main} remain open. 
In the same vein, other classical results in structural graph theory have been studied in this coarse setting in~\cite{GeorgakopoulosP:2025:CoarseGMT}, and the related distance settings and induced settings, 
see for example~\cite{AlbrechtsenHTJKW:2024:CoarseMenger,HendreyNST:2024:CoarseMenger,NguyenSS:2024:CoarseMengerCounterexample} for results related to Menger's Theorem 
and~\cite{AhnGHL:2025:CoarseEP,DujmovicJMM:2025:CoarseEP} for results related to the Erd\H{o}s-P\'{o}sa Theorem.

\paragraph{Structure of the paper}
This paper is structured as follows. 
After introducing the necessary notation in \zcref{sec:prelims}, we prove some necessary lemmas in \zcref{sec:main}. 
In this section, we also formulate the main lemma, and, given its validity, give a proof of the main theorem. 
Then, in \zcref{sec:cobwebs}, we introduce the necessary machinery to prove the main lemma from the previous section, which we then prove in \zcref{sec:mainlemma}.

\section{Preliminaries}
\label{sec:prelims}

We take any standard notions and notation not explained here from~\cite{Diestel:GT5}. 

\paragraph{Basic notions and notation} 
Given an integer~$k$, we write~${[k] \coloneqq \{ i \in \mathbb{Z} \colon 1 \leq i \leq k\}}$ for the set of all positive integers less or equal to~$k$. 

All graphs in this paper are finite and simple. 
Given a graph~$G$, we denote its vertex set by~$V(G)$ and its edge set by~$E(G)$. 
For convenience, slightly abusing the notation, we sometimes refer to the vertex set of a graph~$G$ just by~$G$, in order to simplify some expressions in the cases where there is no danger of confusion. 
In particular, when a function or operator~$f$ takes a set of vertices~$A$ as its input, we may input a subgraph~$H$ instead of~$A$, so $f(H)$ will denote~$f(V(H))$. 
Moreover, if~$H$ is a subgraph of~$G$ and~${A \subseteq V(G)}$, we may also write~$f(H \cup A)$ and~$f(H \cap A)$ to denote~$f(V(H) \cup A)$ and~$f(V(H) \cap A)$, respectively. 

For a graph~$G$ and a set~${A \subseteq V(G)}$, we write~$G[A]$ for the induced subgraph of~$G$ with vertex set~$A$. 
For a vertex~${v \in V(G)}$, we denote by~$N_G(v)$ the \emph{(open) neighbourhood} of~$v$, that is, the set of vertices in~$G$ adjacent to~$v$, and by~$N_G[v]$ the \emph{closed neighbourhood} of~$v$, that is, the set~${N_G(v) \cup \{v\}}$. 
Similarly, for a set~${A \subseteq V(G)}$, 
the \emph{(open) neighbourhood} of~$A$ is the set~${N_G(A) \coloneqq \left(\bigcup_{v \in A} N_G(v) \right) \setminus A}$, 
and the \emph{closed neighbourhood} of~$A$ is the set~${N_G[A] \coloneqq \bigcup_{v \in A} N_G[v]}$. 
For all these notions, we drop the subscript if the ambient graph is clear from context. 

Given two graphs~$G$ and~$H$, we denote by~${G \cap H}$ the \emph{intersection graph}, that is the graph with vertex set~${V(G) \cap V(H)}$ and edge set~${E(G) \cap E(H)}$. 
Note that the intersection graph need not be an induced subgraph of either~$G$ or~$H$. 
Given two graphs~$G$ and~$H$, we denote by~${G \cup H}$ the \emph{union graph}, that is, the graph with vertex set~${V(G) \cup V(H)}$ and edge set~${E(G) \cup E(H)}$. 
Given a graph~$G$ and a set~${A \subseteq V(G)}$, we denote the graph~${G[V(G) \setminus A]}$ obtained by deleting all vertices in~$A$ by~${G-A}$. 
When~${A = \{a\}}$ is a singleton, we instead may also write~${G-a}$. 

We call a family of graphs \emph{intersecting} if the pairwise intersections of the vertex sets of these graphs are non-empty.

\paragraph{Paths and separators} 
We denote a path as a string of its vertices, that is, if we have a set~$\{v_1, \dots, v_k\}$ of~$k$ distinct vertices, then ${v_1 v_2 \dots v_k}$ is the path with vertex set~$\{v_1, \dots, v_k\}$ and edge set~$\{ v_i v_{i+1} \colon i \in [k-1]\}$, where~$k$ is a positive integer. 

An \emph{$(a,b)$-path} is a path whose set of endpoints is equal to~$\{a,b\}$. 
Given a graph~$G$ and sets~${A,B \subseteq V(G)}$ of vertices, an \emph{$(A,B)$-path} is an $(a,b)${\dash}path in~$G$ with~${a \in A}$, ${b \in B}$, and no internal vertex in~${A \cup B}$. 
Given two subgraphs~$H$ and~$H'$ of~$G$, we also call an~$(V(H),V(H'))$-path an~\emph{$(H,H')$-path} instead. 

A set~${S \subseteq V(G)}$ \emph{separates} $A$ from~$B$ if it intersects every $(A,B)$-path. 
When~${A = \{v\}}$ for a single vertex~$v$, we also just say that~$S$ separates~$v$ from~$B$. 
We may also say that~$S$ of~$G$ separates a subgraph~$H$ from~$B$ to mean that~$S$ separates~$V(H)$ from~$B$.

\paragraph{Minors and induced minors} 
Let us recall some terminology regarding minors and induced minors. 

A graph~$H$ is a \emph{minor} of a graph~$G$ if~$H$ is isomorphic to a graph that we can obtain from~$G$ by deleting vertices, deleting edges, and contracting edges. 
We say~\emph{$G$ excludes $H$ as a minor} if~$H$ is not a minor of~$G$, in which case we also say that~$G$ is \emph{$H$-minor-free}. 

A graph~$H$ is an \emph{induced minor} of a graph~$G$ 
if there exists a family~$\{ X_v \colon v \in V(H)\}$ of pairwise vertex-disjoint subsets of $V(G)$ such that~$G[X_v]$ is connected for all~${v \in V(H)}$ and there exists an edge in $G$ between~$X_v$ and~$X_w$ if and only of~${vw \in E(H)}$, for any two distinct ${v,w \in V(H)}$.
We call such a family~$\{ X_v \colon v \in V(H)\}$ and \emph{induced minor model of~$H$ in~$G$}. 
Equivalently, $H$ is an induced minor of $G$ if~$H$ is isomorphic to a graph that we can obtain from~$G$ by deleting vertices and contracting edges (and deleting parallel edges created at each step). 
We say~\emph{$G$ excludes $H$ as an induced minor} if~$H$ is not an induced minor of~$G$, in which case we also say that~$G$ is \emph{$H$-induced-minor-free}. 

For every positive integer~$n$, we denote the complete graph on~$n$ vertices by~$K_n$. 
Note that~$K_n$ is a minor of a graph~$G$ if and only if~$K_n$ is an induced minor of~$G$. 

Let~$v$ be a vertex in~$G$ of degree~$2$. 
The operation of \emph{suppressing}~$v$ is the operation of contracting exactly one edge incident to~$v$. 

Recall that for each integer~${\ell \geq 3}$, we denote by~$W_\ell$ the \emph{$\ell$-wheel}, that is, the graph obtained from a cycle of length~${\ell}$ by adding a universal vertex. 

\begin{observation}
    \label{obs:WheelFromLargeNeighbourhood}
    Let~${\ell \geq 3}$ be an integer, let~$G$ be a graph, and let~$C$ be an induced cycle in~$G$. 
    If some component of~${G - V(C)}$ has at least~$\ell$ neighbours in~$C$, then~$G$ contains~$W_\ell$ as an induced minor. 
\end{observation}

\begin{proof}
    Let~$D$ be a component of~${G-V(C)}$ with at least~$\ell$ neighbours in~$C$ and let~${P \subseteq C}$ be a minimal connected subgraph of~$C$ containing exactly~$\ell$ neighbours of~$D$. 
    Observe that~$P$ is a (not necessarily induced) path whose endpoints are both neighbours of~$D$.
    Then contracting in $G[C \cup D]$ each edge in~$E(D)$ and all but one edge in~$E(C)\setminus E(P)$ and suppressing each vertex of degree~$2$ yields the~$W_\ell$ induced minor. 
\end{proof}

\paragraph{Treewidth} 
A \emph{tree-decomposition} of a graph~$G$ consists of a tree~$T$ and a subtree~$T_v$ of~$T$ for each vertex~$v$ of~$G$ such that for every edge~$uv$ of~$G$, $T_u$ and $T_v$ have a common node.
For each node~$t$ of the tree~$T$, we let~${X_t \coloneqq \{v \colon t \in V(T_v)\}}$ be the \emph{bag} corresponding to~$t$ and define the \emph{width} of the tree-decomposition as the maximum of~${\abs{X_t}-1}$ over the nodes~$t$ of the tree. 
The \emph{treewidth} of~$G$, denoted by~$\tw(G)$, is the minimum of the widths of its tree-decompositions.

We will use the following lemma. 

\begin{lemma}[see \cite{BrandstaedtLS:1999:GraphClasses}]
    \label{lem:SeriesParallelTW}
    Every $K_4$-minor-free graph has treewidth at most~$2$. 
\end{lemma}

\paragraph{Vertex-weighted graphs and balanced separators} 
In this paper, we consider graphs where the vertices are weighted. 
Consider a \emph{weighting}~${\mathsf{w} \colon V(G) \to \mathbb{R}_{\geq 0}}$  of~$G$ with non-negative reals. 
For~${A \subseteq V(G)}$, we write~${\mathsf{w}(A) \coloneqq \sum_{v \in A} \mathsf{w}(v)}$ for the \emph{$\mathsf{w}$-weight of~$A$}. 
Similarly, for a subgraph~$H$ of~$G$, we write~$\mathsf{w}(H)$ for~$\mathsf{w}(V(H))$. 
We say~$\mathsf{w}$ is \emph{non-trivial} if~${\mathsf{w}(G) > 0}$, and \emph{trivial}, otherwise. 
We call a subgraph~$H$ of~$G$ \emph{heavy} (with respect to~$\mathsf{w}$) if~${\mathsf{w}(H) > \nicefrac{1}{2} \cdot \mathsf{w}(G)}$. 

Given a graph~$G$ and a weighting~$\mathsf{w}$ of~$G$, we call a set~${S \subseteq V(G)}$ of vertices a \emph{$\mathsf{w}${\dash}balanced separator of~$G$} if no component of~${G - S}$ is heavy. 
Note that for the trivial weighting~$\mathsf{c}_0$ of a graph~$G$, where~${\mathsf{c}_0(v)=0}$ for every~$v$, the empty set is a~$\mathsf{c}_0$-balanced separator, and for the weighting~$\mathsf{c}_1$ of~$G$ that assigns weight~$1$ to each vertex of~$G$, a $\mathsf{c}_1$-balanced separator is exactly a balanced separator as defined in the introduction. 

For convenience, we often want to look at \emph{normal} weightings, that is weightings with~${\mathsf{w}(G) = 1}$. 
Clearly, given any graph~$G$ with a non-trivial weighting~$\mathsf{w}$ of~$G$, we can define a normal weighting~$\mathsf{w}'$ by setting~${\mathsf{w}'(v) = \nicefrac{\mathsf{w}(v)}{\mathsf{w}(G)}}$ for all~${v \in V(G)}$. 
It is now easy to see that a set~${S \subseteq V(G)}$ is a $\mathsf{w}$-balanced separator if and only if it is a $\mathsf{w}'$-balanced separator. 

If a set~${S \subseteq V(G)}$ is not a $\mathsf{w}$-balanced separator for some normal weighting~$\mathsf{w}$, then there exists a heavy component~$D$ of~${G-S}$. 
Note that this component is unique, as if~${D' \neq D}$ would also be such a component, then since~$D$ and~$D'$ are disjoint, we get~${\mathsf{w}(G) \geq \mathsf{w}(D) + \mathsf{w}(D') > \nicefrac{1}{2} + \nicefrac{1}{2} = 1}$, contradicting the normality of~$\mathsf{w}$. 

We make the following standard observation, for which we include a proof for the sake of completeness. 

\begin{observation}
    \label{obs:TreeBalancedSep}
    Every tree~$T$ with a weighting~$\mathsf{w}$ has a $\mathsf{w}$-balanced separator of size~$1$. 
\end{observation}

\begin{proof}
    Without loss of generality, we assume~$\mathsf{w}$ is normal. 
    If there exists an edge~$tu$ of~$T$ such that both components of the forest~$T_{tu}$ obtained from~$T$ by deleting the edge~$tu$ are not heavy, then~$\{t\}$ is a $\mathsf{w}$-balanced separator. 
    So we may assume that for each edge~$tu$ of~$T$, exactly one component of~${T_{tu}}$ is heavy. 
    Orienting every edge towards the node in the unique heavy component of~${T_{tu}}$ yields a sink~$s$ in that orientation. 
    Now clearly $\{s\}$ is a $\mathsf{w}$-balanced separator. 
\end{proof}

We will also use the fact that the treewidth of a graph bounds the size of a smallest balanced separator. 

\begin{lemma}[{\cite[Lemma 7.19]{CyganFKLMPPS:2015:ParameterizedAlgorithms}}]
    \label{cor:TreeWidthBalancedSep}
    Every graph~$G$ with a weighting~$\mathsf{w}$ has a $\mathsf{w}$-balanced separator of size at most~$\tw(G) + 1$. 
\end{lemma}

\section{The main theorem}
\label{sec:main}

\subsection{Thinning out balanced separators}

Let us first prove a very useful lemma about turning a balanced separator into a smaller one. 

\begin{lemma}
    \label{lem:modifyBalancedSep}
    Let~$G$ be a graph, let~$\mathsf{w}$ be a normal weighting of~$G$ and let~${X, Y \subseteq V(G)}$, and let~${Z \subseteq X}$.  
    Suppose that~${X \cup Y}$ is a $\mathsf{w}$-balanced separator of~$G$ but that~$Z$ is not. 
    Let~$B$ be the heavy component of~${G - Z}$. 
    Then~${S \coloneqq (N[B] \cap X) \cup Y}$ is a $\mathsf{w}$-balanced separator of~$G$. 
\end{lemma}

\begin{proof}
    Consider a component~$D$ of~${G - S}$. 
    We first observe that~$D$ is either a subgraph of~$B$ or vertex-disjoint from~$B$. 
    Indeed, otherwise there exists~${v \in V(D \cap N(B))}$, implying~${v \in N[B] \cap Z}$, which would result in~${v \in S}$. 
    
    If~$D$ is vertex-disjoint from~$B$, then ${\mathsf{w}(D) \leq 1 - \mathsf{w}(B) < \nicefrac{1}{2}}$. 
    So suppose~$D$ is a subgraph of~$B$, and hence of~${B - S = B - (X \cup Y)}$. 
    Since~$D$ is connected, it is a subgraph of a component of~${G - (X \cup Y)}$, 
    and since~${X \cup Y}$ is a $\mathsf{w}$-balanced separator of~$G$, we conclude that~$\mathsf{w}(D) \leq \nicefrac{1}{2}$. 
\end{proof}

\subsection{Gy\'{a}rf\'{a}s cycles}

We need the following Gy\'{a}rf\'{a}s path argument that follows from the proof of~{\cite[Lemma~5.3]{ChudnovskyPPT:2024:ApproxMWIS}}. 

\begin{lemma}
    \label{lem:GyarfasPath}
    For every graph~$G$ and weighting~$\mathsf{w}$ of~$G$ there exists an induced path~$P$ in~$G$ such that~$N[P]$ is a $\mathsf{w}$-balanced separator of~$G$.     
\end{lemma}

Note that \zcref{lem:modifyBalancedSep,lem:GyarfasPath} are already sufficient to prove a simpler case, where instead of excluding a wheel as an induced minor, 
we exclude a \emph{fan}, that is, a graph obtained from adding a universal vertex to a path. 
We include a proof of this simple argument in the appendix of this paper, see \zcref{thm:fanfree}. 

Using these two lemmas, we extend the argument to find a balanced separator that is dominated by either at most two vertices, or by an induced cycle and one other vertex. 
Note that this then also proves the simpler case where instead of excluding a wheel as an induced minor we exclude a cycle as an induced minor. 

\begin{lemma}
    \label{lem:Gyarfas(Cycle+Vertex)}
    For every graph~$G$ and weighting~$\mathsf{w}$ of~$G$ there exists either 
    \begin{itemize}
        \item an induced cycle~$C$ in~$G$ of length at least~$4$ and a vertex~${p \in V(G - C)}$ such that~$N[C \cup \{p\}]$ is a $\mathsf{w}$-balanced separator of~$G$, or
        \item vertices~${p,q \in V(G)}$ such that~$N[\{p,q\}]$ is a $\mathsf{w}$-balanced separator of~$G$. 
    \end{itemize}
\end{lemma}

\begin{proof}
    Without loss of generality, we assume~$\mathsf{w}$ is normal. 
    
    Let~$P = p_1 p_2 \dots p_k$ be a minimal path as in \zcref{lem:GyarfasPath} and let~${p \coloneqq p_k}$. 
    If~${k = 1}$, then setting~${q \coloneq p_k}$ yields that~${N[P] = N[\{p,q\}]}$ is a $\mathsf{w}$-balanced separator of~$G$ by \zcref{lem:GyarfasPath}, so we may assume~${k > 1}$. 
    By the minimality of~$P$ there is a unique heavy component~$B$ of~${G - N[P - p]}$. 
    Applying \zcref{lem:modifyBalancedSep} with~${X = Z = N[P - p]}$ and~${Y = N[p]}$, we infer that the set ${S \coloneqq N(B) \cup N[p]}$ is a $\mathsf{w}$-balanced separator of~$G$. 
    
    Let~${Q = p_{a} p_{a+1} \dots p_{b}}$ be a minimal subpath of~${P - p}$ such that~${N(B) \subseteq N(Q)}$. 
    If~${a = b}$, then by setting~${q \coloneq p_a}$, we observe that~${N[\{p,q\}]}$ is a $\mathsf{w}$-balanced separator of~$G$ since ${S \subseteq N[\{p,q\}]}$, so we may assume~${p_a \neq p_b}$. 
    By the minimality of~$Q$, we observe that~$p_a$ has a neighbour~$x$ in~${N(B)}$ that is not a neighbour of any~$p_i$ for~${a < i \leq b}$ and that~$p_b$ has a neighbour~$y$ in~${N(B)}$ that is not a neighbour of any~$p_i$ for~${a \leq i < b}$. 
    Let~$R$ be a shortest path between~$x$ and~$y$ in~${G[B \cup \{x,y\}]}$. 
    Now the cycle~${C \coloneqq Q \cup R}$ is induced and~${S \subseteq N[C \cup \{p\}]}$, so we have that~${N[C \cup \{p\}]}$ is the desired $\mathsf{w}$-balanced separator of~$G$. 
\end{proof}

\subsection{The main theorem}
We prove a slightly more general statement than \zcref{thm:main}, by using weighted graphs and giving an explicit bound on the number of vertices that dominate the balanced separator. 

\begin{theorem}[store=mainthmweighted]
    \label{thm:mainWeighted}
    For every integer~${\ell \geq 3}$ and every $W_\ell$-induced-minor-free graph~$G$ with weighting~$\mathsf{w}$ of~$G$ there is a $\mathsf{w}$-balanced separator of~$G$ that either is dominated by at most~$\ell$ vertices or has size at most~${(\ell-1)^2}$. 
\end{theorem}

To prove this theorem, we first find a balanced separator of the form~$N[C \cup \{p\}]$ or~$N[\{p,q\}]$ as in \zcref{lem:Gyarfas(Cycle+Vertex)}. 
In the latter case, we are done. 
In the former case, if some component~$B$ of~${G - V(C)}$ is heavy, then using \zcref{lem:modifyBalancedSep}, we find that~${N[N(B) \cup \{p\}]}$ is a balanced separator, which is dominated by at most~$\ell$ vertices by \zcref{obs:WheelFromLargeNeighbourhood}. 
If no component of~${G - V(C)}$ is heavy, we contract each component to a single vertex and modify the weighting accordingly. 
For this case, we will use the following lemma. 

\begin{lemma}[store=mainlemma]
    Let~${\ell \geq 4}$ be an integer, let~$G$ be a connected $W_\ell$-induced-minor-free graph, and let~$\mathsf{w}$ be a normal weighting of~$G$. 
    Let~$C$ be an induced cycle in~$G$ of length at least~$4$ such that~${J \coloneqq V(G) \setminus V(C)}$ is independent. 
    If~${\mathsf{w}(v) \leq \nicefrac{1}{2}}$ for all~${v \in J}$, then there is a $\mathsf{w}$-balanced separator~${S \subseteq V(C)}$ of~$G$ of size at most~${(\ell-1)^2}$. 
    \label{lem:NoBigComponents}
\end{lemma}

We will only prove this lemma in \zcref{sec:mainlemma} after introducing the necessary machinery in \zcref{sec:cobwebs}. 
Let us conclude this section with a formal proof of the main theorem, assuming the validity of \zcref{lem:NoBigComponents}. 

\begin{proof}[Proof of \zcref{thm:mainWeighted}]
    Without loss of generality, we assume that~$\mathsf{w}$ is normal. 
    If no component of~$G$ is heavy, there is nothing to show, so we may assume that~$G$ has a (unique) heavy component~$H$. 
    Hence, we may assume that~$G$ is connected, since every $\mathsf{w}_H$-balanced separator of~$H$, where $\mathsf{w}_H$ is the restriction of $\mathsf{w}$ to~$H$, is a $\mathsf{w}$-balanced separator of~$G$. 

    First observe that if~${\ell = 3}$, then~$G$ is $K_4$-minor-free, so by \zcref{lem:SeriesParallelTW} and \zcref{cor:TreeWidthBalancedSep}, $G$ has a $\mathsf{w}$-balanced separator of size~$3$. 
    So we may assume that~${\ell \geq 4}$. 
    
    We apply \zcref{lem:Gyarfas(Cycle+Vertex)}. 
    If we find a $\mathsf{w}$-balanced separator of~$G$ of the form~$N[\{p,q\}]$, then there is nothing else to show. 
    So we may assume we find an induced cycle~$C$ and a vertex~${p \in V(G - V(C))}$ such that~$N[C \cup \{p\}]$ is a $\mathsf{w}$-balanced separator of~$G$. 
    
    First, suppose that there is a heavy component~${B}$ of~${G - V(C)}$. 
    Then, by \zcref{lem:modifyBalancedSep} (with~${X = N[C]}$, ${Y = N[p]}$, and~${Z = C}$), we obtain that the set ${S \coloneqq N[N(B) \cup \{p\}]}$ is a $\mathsf{w}$-balanced separator of~$G$. 
    Since~$G$ is $W_\ell$-induced-minor-free, ${\abs{N(B)} < \ell}$ by \zcref{obs:WheelFromLargeNeighbourhood}, and we have that~$S$ is dominated by at most~$\ell$ vertices. 

    Secondly, suppose that no component of~${G - V(C)}$ is heavy. 
    Let~$G'$ be the graph obtained from~${G}$ by contracting each component of~${G - V(C)}$ to a single vertex. 
    Formally, we consider vertices of~$G'$ to be either vertices of~$C$ or the vertex sets of the components of~${G - V(C)}$. 
    Let~$\mathsf{w}'$ be the normal weighting of~$G'$ defined by 
    \[
        \mathsf{w}'(v) = 
        \begin{cases}
            \mathsf{w}(v) 
                & \text{if } v \in V(C)\,, \\
            \sum_{x \in v} \mathsf{w}(x) 
                & \text{if } G[v] \textnormal{ is a component of } G - V(C) \,.
        \end{cases}
    \]
    Now, $G'$ is connected and~$C$ is an induced cycle in~$G$ such that~${J \coloneqq V(G') \setminus V(C)}$ is independent. 
    Moreover, ${\mathsf{w}'(v) \leq \nicefrac{1}{2}}$ for all~${v \in J}$. 
    Hence, by \zcref{lem:NoBigComponents}, there is a $\mathsf{w}'$\nobreakdash-balanced separator~$S$ of~$G'$ with~${S \subseteq V(C) \subseteq V(G)}$ 
    and~${\abs{S} \leq (\ell -1)^2}$. 
    Now by the definition of~$\mathsf{w}'$, it immediately follows that~$S$ is also a $\mathsf{w}$-balanced separator of~$G$. 
\end{proof}

\section{Cobwebs}
\label{sec:cobwebs}

In this section, we introduce concepts that will be useful to prove \zcref{lem:NoBigComponents} in the subsequent section. 

We say a graph~$G$ is a \emph{cobweb} if 
\begin{itemize}
    \item there is an induced cycle~$C$ in~$G$ such that~${I \coloneqq V(G) \setminus V(C)}$ is a non-empty independent set in~$G$, 
    \item $G$ has minimum degree at least~$2$, and 
    \item $N(u) \not\subseteq N(v)$ for every pair~${u,v}$ of distinct vertices in~$I$. 
\end{itemize}
We call the pair~$(C,I)$ a \emph{presentation} of~$G$. 
See \zcref{fig:CobwebsAndBasicHoles}(a) for an example. 

\begin{figure}[htbp]
    \centering
    \includegraphics[scale=0.5]{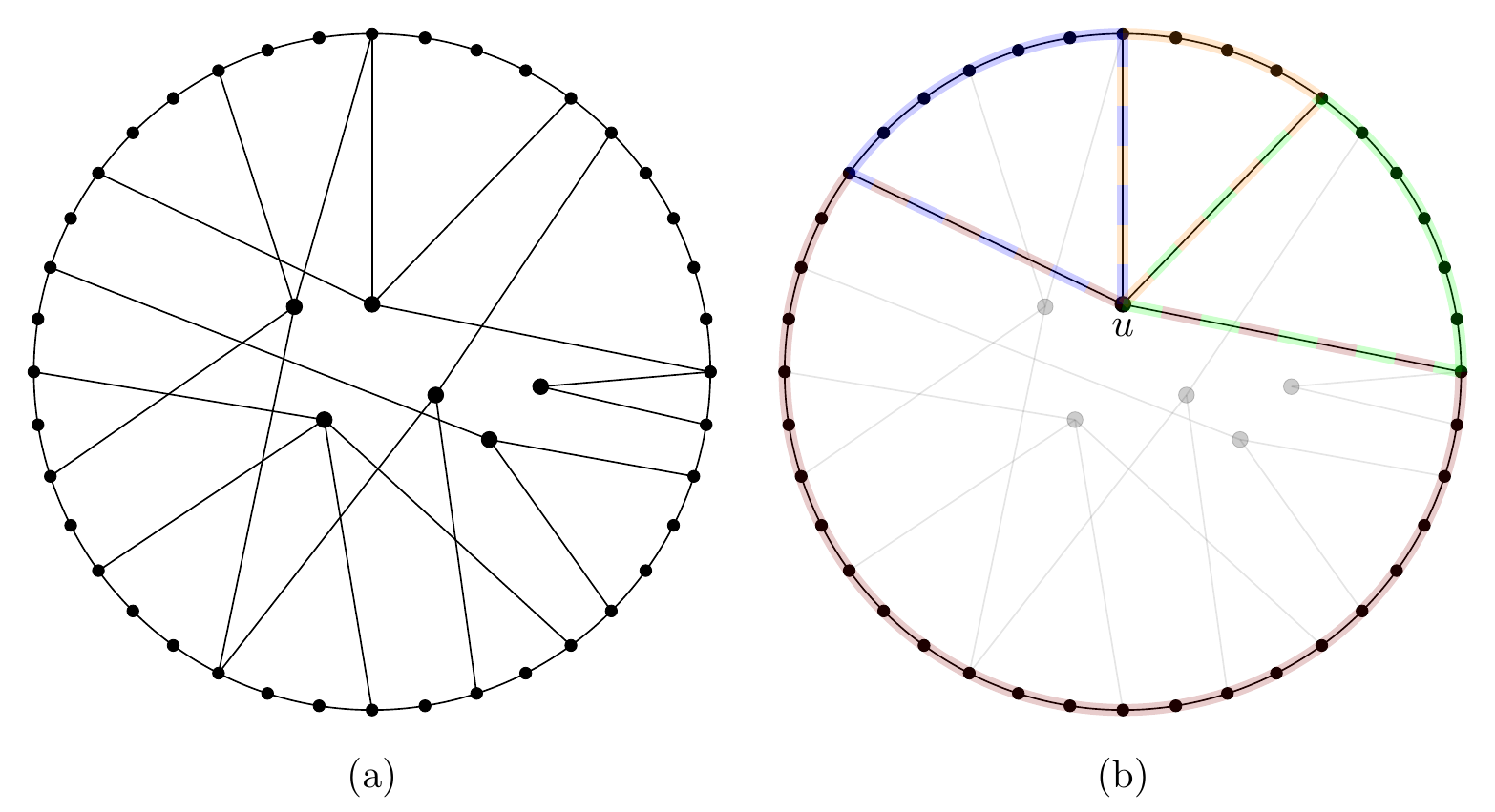}\\
    \vspace*{-1.5cm}
    \includegraphics[scale=0.5]{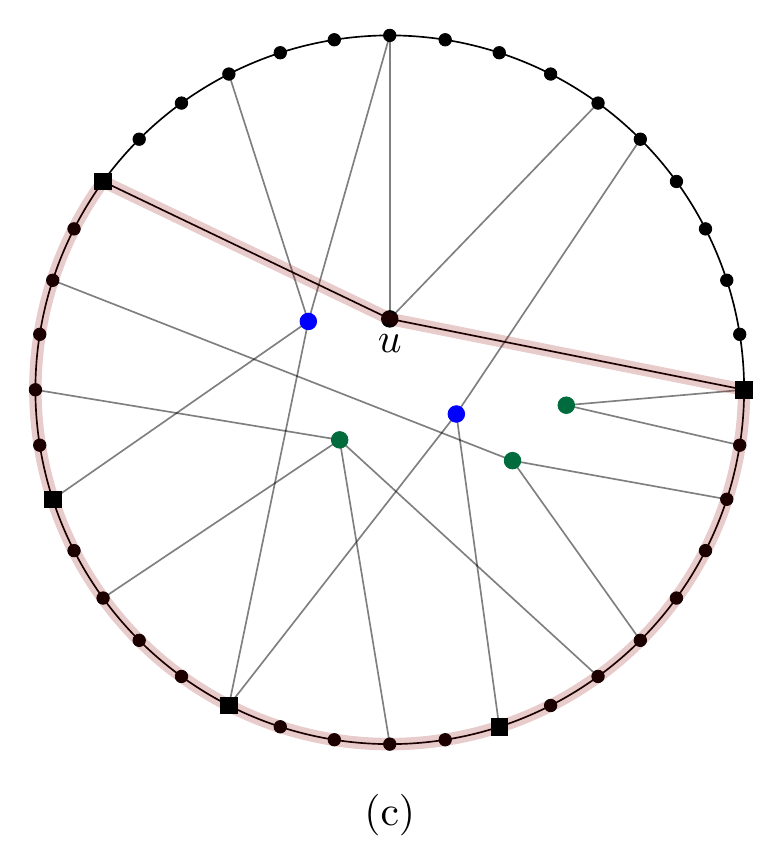}
    \caption{(a) An example of a cobweb. 
    (b) A vertex ${u \in I}$ and the four wedges anchored in~$u$. (c) For one of the wedges~$W$, vertices coloured {\color{cadmiumgreen}\textbf{green}} are the vertices attached to~$W$, vertices coloured  {\color{blue}\textbf{blue}} are the vertices not attached to~$W$, and the square vertices represent the barrier~$\barrier{W}$. }
    \label{fig:CobwebsAndBasicHoles}
\end{figure}

Let~$G$ be a cobweb with presentation~${(C,I)}$. 
Let us call a (not necessarily induced) cycle~$W$ in~$G$ a \emph{wedge} of~${(G,C,I)}$ if~$W$ contains exactly one vertex~$v$ of~$I$, together with exactly two vertices of~${N_G(v)}$. 
We say~$W$ is \emph{anchored} in~$v$ and~$v$ is the \emph{anchor} of~$W$. 
So in particular, ${W-v}$ is a path in~$G$ that is contained in~$C$ and~${N(v) \cap V(W-v)}$ is exactly the set of endpoints of the path~${W-v}$. 
Moreover, if~$W$ has length at least~$4$, then~${W-v}$ is a path of length at least~$2$, and if~$W$ has length~$3$, then~${W - v}$ is a single edge. 
In the latter case, we say~$W$ is \emph{trivial}. 
Note that since~$G$ has minimum degree at least~$2$, every vertex~$v$ is the anchor of exactly~$\deg(v)$ many wedges, see also \zcref{fig:CobwebsAndBasicHoles}(b). 
If~${V(C) \subseteq V(C)}$, we say that~$W$ is \emph{co-trivial}. 
This is the case when the anchor~$v$ of~$W$ has degree exactly~$2$ and the unique other wedge anchored in~$v$ is trivial. 
Note that every non-co-trivial wedge is an induced subgraph of~$G$, and the only edge that prevents a co-trivial wedge from being an induced subgraph is the unique edge of~$C$ not contained in the wedge. 

As observed above, the intersection of a wedge~$W$ with~$C$ is a path. 
So, in particular, if~$W$ is not co-trivial, there is a unique component of~${G - V(W)}$ that contains vertices of~$C$. 
We define the \emph{barrier} ${\barrier{W}}$ of a non-co-trivial wedge~$W$ to be the set of neighbours~$C$ of the unique component of~${G - V(W)}$ that contains vertices of~$C$. 
Observe that the barrier of $W$ consists precisely of the endpoints of the path~${W \cap C}$, along with the internal vertices of this path that have a common neighbour in~$I$ with some vertex in~${V(C) \setminus V(W)}$. 
In particular, the anchor of $W$ does not belong to the barrier. 

For a co-trivial wedge~$W$, we define the \emph{barrier} ${\barrier{W}}$ of~$W$ to be the two neighbours of the anchor of~$W$. 
We say a vertex~${u \in I \setminus V(W)}$ \emph{attaches to} a wedge~${W}$ if~${N(u) \subseteq V(W)}$. 
Observe that if~$W$ is a trivial wedge, then no vertex in~$I$ attaches to~$W$, and if~$W$ is co-trivial, then every vertex in~$I$ except the anchor of~$W$ attaches to~$W$. 
See \zcref{fig:CobwebsAndBasicHoles}(c) for an example of attaching and the barrier. 

\begin{observation}
    \label{obs:CobwebAttaches}
    Let~$G$ be a cobweb with presentation~$(C,I)$, let~$W$ be a wedge of~$(G,C,I)$, and let~${u \in I}$ such that~$u$ does not attach to~$W$.
    Then, the barrier~$\barrier{W}$ separates~$u$ from~${W \cap C}$. 
\end{observation}

\begin{proof} 
    If~$W$ is trivial, then~$W \cap C$ is contained in~$\barrier{W}$, so there is nothing to show. 
    If~$u$ is the anchor of~$W$ (in particular if~$W$ of co-trivial), then the claim is true since both endpoints of the path~${W \cap C}$ are contained in $\barrier{W}$. 
    Otherwise, let~$D$ denote the unique component of~${G - V(W)}$ that contains vertices of~$C$. 
    By definition, $\barrier{W}$ separates~${V(W \cap C)}$ from~$V(D)$. 
    Since~${u}$ does not attach to~$W$, by definition~$u$ has a neighbour in~${V(C) \setminus V(W)}$ and hence~${u \in V(D)}$. 
\end{proof}

We note that the converse of \zcref{obs:CobwebAttaches} does not hold, i.e.~${\barrier{W}}$ may separate~$u$ from~${V(W \cap C)}$ even though~$u$ is attached to~$W$  
(indeed, this is the case whenever ${N(u)\subseteq \barrier{W}}$).

\medskip

We call a function~${\wedgeselectionnull}$ mapping each~${v \in I}$ to a wedge~${\wedgeselection{v}}$ of~${(G,C,I)}$ anchored in~$v$ for each~${v \in I}$ a \emph{wedge-selection for $(G,C,I)$}. 
Slightly abusing the notation for the barrier, we write~$\barrier{v}$ for the barrier~$\barrier{\wedgeselection{v}}$. 

Given a wedge-selection~$\wedgeselectionnull$ for~$(G,C,I)$ and~${u, v \in I}$, 
we say that~${\wedgeselection{u}}$ \emph{improves} ${\wedgeselection{v}}$ if~${\wedgeselection{u} - u}$ is a proper subgraph of~${\wedgeselection{v} - v}$. 
For~${u,v \in I}$, we write~${u \prec v}$ if~${\wedgeselection{u}}$ improves~${\wedgeselection{v}}$. 
Moreover, we write~${u \preceq v}$ if either~${u \prec v}$ or~${u = v}$. 
This defines a partial order, and we call the minimal elements of this partial order \emph{good} (with respect to~$\wedgeselectionnull$). 
Clearly, good elements exist. 

\begin{lemma}
    \label{obs:AttachPairOrImprove}
    Let~$G$ be a cobweb with presentation~${(C,I)}$ and let~$\wedgeselectionnull$ be a wedge-selection for~${(G,C,I)}$. 
    Let~${u,v \in I}$ be distinct. 
    If~$u$ attaches to~$\wedgeselection{v}$, then either~${\wedgeselection{u}}$ improves ${\wedgeselection{v}}$, or $v$ attaches to~${\wedgeselection{u}}$ as well, but not both. 
\end{lemma}

\begin{proof}
    Suppose~$u$ attaches to~$\wedgeselection{v}$, so in particular, $\wedgeselection{v}$ is non-trivial. 
    We show that~${\wedgeselection{u}}$ improves ${\wedgeselection{v}}$ if and only if~$v$ does not attach to~$\wedgeselection{u}$. 
    
    First, suppose that~$\wedgeselection{u}$ improves~$\wedgeselection{v}$, so~${\wedgeselection{u}-u}$ is a proper subgraph of~${\wedgeselection{v} - v}$. 
    Then, in particular, one of the endpoints of the path~${\wedgeselection{v} - v}$ is not contained in~${\wedgeselection{u}-u}$, so~${N(v) \not\subseteq V(\wedgeselection{u})}$. 
    
    If on the other hand~$\wedgeselection{u}$ does not improve~$\wedgeselection{v}$, then either~${\wedgeselection{u} - u = \wedgeselection{v} - v}$ or~${\wedgeselection{u}-u}$ contains an internal vertex from the path~${Q}$ obtained by deleting from~$C$ the internal vertices of~${\wedgeselection{v} - v}$. 
    If ${\wedgeselection{u} - u = \wedgeselection{v} - v}$, then, since~$N(u) \subseteq V(\wedgeselection{v})$, we have that~$N(u) \subseteq N(v)$, a contradiction to the fact that~$G$ is a cobweb with presentation~${(C,I)}$.
    If ${\wedgeselection{u}-u}$ contains an internal vertex of~$Q$, then, since $u$ attaches to~$\wedgeselection{v}$, the only wedge anchored in~$u$ that contains an internal vertex of~$Q$ contains the whole path~$Q$, hence this wedge has to be $\wedgeselection{u}$, implying that~$v$ attaches to~$\wedgeselection{u}$. 
\end{proof}

Now let us analyse the structure we have looking at two vertices that pairwise attach to the other's selected wedges, see \zcref{fig:MoreCobwebs}(a) for a reference. 

\begin{figure}[htbp]
    \centering
    \includegraphics[scale=0.5]{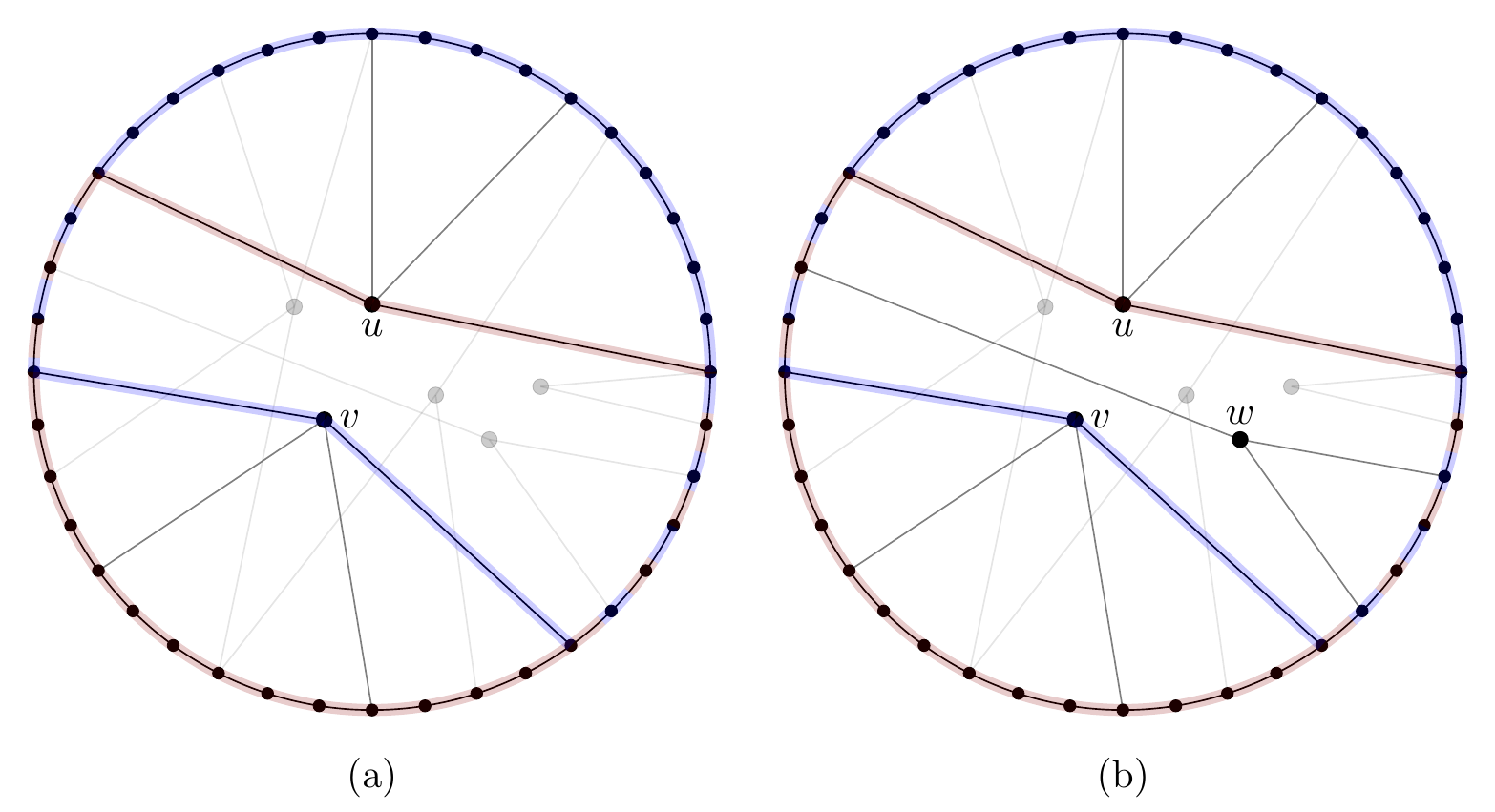}
    \caption{(a) A cobweb with a wedge $\wedgeselection{u}$ (in red) and a wedge $\wedgeselection{v}$ (in blue). Here, $u$ attaches to~$\wedgeselection{v}$ and~$v$ attaches to~$\wedgeselection{u}$. \\
    (b) The vertex~$w$ attaches to both~$\wedgeselection{u}$ and~$\wedgeselection{v}$ and connects both components of~${\wedgeselection{u} \cap \wedgeselection{v}}$. Regardless which wedge anchored in~$w$ is selected as~$\wedgeselection{w}$, it will improve either~$\wedgeselection{u}$ or~$\wedgeselection{w}$ (or both).}
    \label{fig:MoreCobwebs}
\end{figure}

\begin{lemma}
    \label{lem:attachingpair}
    Let~$G$ be a cobweb with presentation~${(C,I)}$ and let~$\wedgeselectionnull$ be a wedge-selection for~${(G,C,I)}$. 
    Then for all~${u, v, w \in I}$, the following statements hold. 
    \begin{enumerate}[label=\textnormal{(\alph*)}]
        \item\label{itema} The graph ${\wedgeselection{u} \cap \wedgeselection{v}}$ has at most two components. 
        \item\label{itemb} The graph ${\wedgeselection{u} \cap \wedgeselection{v}}$ has exactly two components, if and only if~$u$ attaches to~$\wedgeselection{v}$ and~$v$ attaches to~$\wedgeselection{u}$.
        \item\label{itemc}
        If ${\wedgeselection{u} \cap \wedgeselection{v}}$ has exactly two components and $w$ attaches to both~$\wedgeselection{u}$ and~$\wedgeselection{v}$, then 
        \begin{itemize}
            \item if~$w$ has a neighbour in each component of~$\wedgeselection{u} \cap \wedgeselection{v}$, then~$\wedgeselection{w}$ improves at least one of~$\wedgeselection{u}$ and~$\wedgeselection{v}$, and
            \item if~$w$ has neighbours in only one component of~$\wedgeselection{u} \cap \wedgeselection{v}$, then~$\wedgeselection{w}$ either improves both~$\wedgeselection{u}$ and~$\wedgeselection{v}$ or none of them. 
        \end{itemize}
    \end{enumerate}
\end{lemma}

\begin{proof}
    We first prove \ref{itema} and \ref{itemb} simultaneously. 
    Suppose that~$u$ attaches to~${\wedgeselection{v}}$ and~$v$ attaches to~${\wedgeselection{u}}$, so neither~$\wedgeselection{u}$ nor~$\wedgeselection{v}$ are trivial wedges. 
    Then, every vertex of~$C$ belongs to~${\wedgeselection{u} \cup \wedgeselection{v}}$. 
    For both~${x \in \{u,v\}}$, let~${Q_x}$ denote the the path obtained by deleting from~$C$ the internal vertices of~${\wedgeselection{x} - x}$ (note that~$Q_x$ is indeed a path since~$\wedgeselection{x}$ is non-trivial). 
    Note that~$Q_u$ and~$Q_v$ are edge-disjoint path of length at least~$1$. 
    Every edge of~$C$ is contained in exactly one of~${E(\wedgeselection{u} \cap \wedgeselection{v})}$, $E(Q_u)$, or~$E(Q_v)$. 
    Hence, ${\wedgeselection{u} \cap \wedgeselection{v}}$ is the graph obtained from~$C$ by deleting the edges of the path~$Q_u$, the edges of the path~$Q_v$, and all internal vertices of either~$Q_u$ or~$Q_v$. 
    From this construction, we observe that ${\wedgeselection{u} \cap \wedgeselection{v}}$ has exactly two components. 
    
    Conversely, without loss of generality, let us assume that~$u$ does not attach to~$\wedgeselection{v}$. 
    Hence, either~${u = v}$ and~$\wedgeselection{u} \cap \wedgeselection{v} = \wedgeselection{u}$ is connected, or~${u \neq v}$ and there exists an~${x \in N(u) \cap V(C - \wedgeselection{v})}$. 
    So~${\wedgeselection{u} \cap \wedgeselection{v} \subseteq C - x}$. 
    Since all of~${C - x}$, ${\wedgeselection{u} - u}$, and~${\wedgeselection{v} - v}$ are paths, we have that~${\wedgeselection{u} \cap \wedgeselection{v}}$, if non-empty, is a path as well (by the Helly property of subpaths of a path). 
    Hence, ${\wedgeselection{u} \cap \wedgeselection{v}}$ has at most one component. 

    Thus, either~${\wedgeselection{u} \cap \wedgeselection{v}}$ has exactly two components (if~$u$ attaches to~${\wedgeselection{v}}$ and~$v$ attaches to~${\wedgeselection{u}}$), or ${\wedgeselection{u} \cap \wedgeselection{v}}$ has at most one component (otherwise). 
    This proves \ref{itema} and \ref{itemb}. 
    
    For \ref{itemc}, suppose~${\wedgeselection{u} \cap \wedgeselection{v}}$ has two distinct components and $w$ attaches to both~$\wedgeselection{u}$ and~$\wedgeselection{v}$, so both~$\wedgeselection{u}$ and~$\wedgeselection{v}$ are non-trivial. 
    
    First suppose that~$w$ has a neighbour in both components of~${\wedgeselection{u} \cap \wedgeselection{v}}$. 
    We distinguish two cases. 
    Either~${\wedgeselection{w} - w}$ is contained in one component of~${\wedgeselection{u} \cap \wedgeselection{v}}$, in which case neither~$u$ nor~$v$ attach to~$\wedgeselection{w}$ and, hence,~$\wedgeselection{w}$ improves both~$\wedgeselection{u}$ and~$\wedgeselection{v}$ by \zcref{obs:AttachPairOrImprove}. 
    Or the path~${\wedgeselection{w} - w}$ has an endpoint in each component of~${\wedgeselection{u} \cap \wedgeselection{v}}$, in which case exactly one of~$u$ or~$v$ attaches to~$\wedgeselection{w}$, see \zcref{fig:CobwebsGoodPair}(b), and hence~$\wedgeselection{w}$ improves exactly one of~$\wedgeselection{u}$ or~$\wedgeselection{v}$, again by \zcref{obs:AttachPairOrImprove}. 
    
    Second, suppose all neighbours of~$w$ are in the same component of~${\wedgeselection{u} \cap \wedgeselection{v}}$, then, as before, 
    either~${\wedgeselection{w} - w}$ is contained in one component of~${\wedgeselection{u} \cap \wedgeselection{v}}$, 
    in which case~$\wedgeselection{w}$ improves both~$\wedgeselection{u}$ and~$\wedgeselection{v}$, or both~$u$ and~$v$ attach to~$\wedgeselection{u}$, and therefore~$\wedgeselection{w}$ improves neither~$\wedgeselection{u}$ nor~$\wedgeselection{v}$ by \zcref{obs:AttachPairOrImprove}. 
\end{proof}

We call a pair~$\{u,v\}$ of (not necessarily distinct) good vertices in~$I$ a \emph{good pair}. 
As a direct consequence of \zcref{lem:attachingpair}, we observe that if~${\{u,v\}}$ is a good pair, then every~${w \in I}$ that attaches to both~$\wedgeselection{u}$ and~$\wedgeselection{v}$ has neighbours in only one component of~${\wedgeselection{u} \cap \wedgeselection{v} \cap C}$. 
Note that unless~${u = v}$, we have that~${\wedgeselection{u} \cap \wedgeselection{v} \subseteq C}$, so the additional intersection with~$C$ guarantees that each component is a path. 

In general, there are three kinds of good pairs~$\{u,v\}$:
\begin{enumerate}
    \item ${u = v}$, 
    \item ${u \neq v}$ and~$u$ does not attach to~$\wedgeselection{v}$ and vice versa (see \zcref{fig:KindsOfGoodPairs}(ii)), or
    \item ${u \neq v}$, ${u}$ attaches to~$\wedgeselection{v}$, and~${v}$ attaches to~$\wedgeselection{u}$  (see \zcref{fig:KindsOfGoodPairs}(iii)). 
\end{enumerate}
\begin{figure}[htbp]
    \centering
    \includegraphics[scale=0.5]{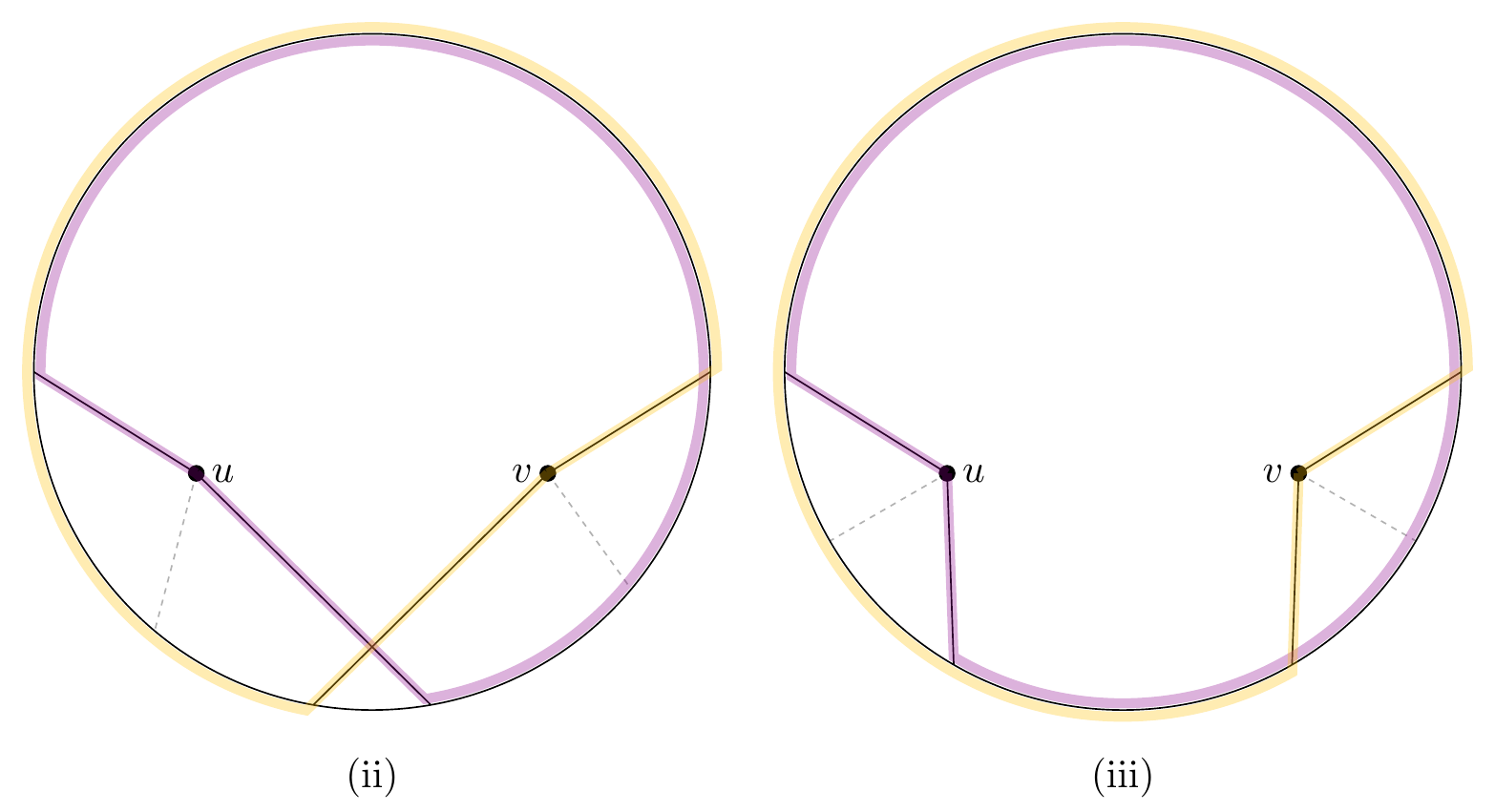}
    \caption{(ii) A good pair of the second kind.\\
    (iii) A good pair of the third kind. }
    \label{fig:KindsOfGoodPairs}
\end{figure}
Note that in pairs of the second kind, $\wedgeselection{u} \cap \wedgeselection{v}$ has at most one component and in pairs of the third kind, ${\wedgeselection{u} \cap \wedgeselection{v}}$ has two components and no vertex~${w \in I}$ that attaches to both~$\wedgeselection{u}$ and~$\wedgeselection{v}$ has neighbours in both components of $\wedgeselection{u} \cap \wedgeselection{v}$. 

Whenever we have a good pair~${\{u,v\}}$, then every component of~${\wedgeselection{u} \cap \wedgeselection{v} \cap C}$ can be separated from the rest of the cycle by deleting the barriers of the respective wedges, as we show in the next lemma. 

\begin{lemma}
    \label{lem:goodpair1}
    Let~$G$ be a cobweb with presentation~${(C,I)}$ and let~$\wedgeselectionnull$ be a wedge-selection for~${(G,C,I)}$. 
    If~${\{u,v\} \subseteq I}$ is a good pair, then each component~$P$ of~${\wedgeselection{u} \cap \wedgeselection{v} \cap C}$ is separated from~${C - V(P)}$ in~$G$ by the union of the two barriers~${\barrier{u} \cup \barrier{v}}$.
\end{lemma}

\begin{proof}
    Let~$P$ be a component of~${\wedgeselection{u} \cap \wedgeselection{v} \cap C}$. 
    If some~${w \in I}$ has a neighbour in~$P$ and a neighbour in~${C - V(\wedgeselection{u} \cap \wedgeselection{v})}$, then~$w$ does not attach to~$\wedgeselection{x}$ for at least one~${x \in \{u,v\}}$, so by \zcref{obs:CobwebAttaches}, ${\barrier{x}}$ separates~$w$ from~$\wedgeselection{x}$ in~$G$. 
    In particular, if~${\wedgeselection{u} \cap \wedgeselection{v}}$ admits a single component, we are done. 
    So suppose~$P$ and~$Q$ are the two components (see \zcref{lem:attachingpair}) of~${\wedgeselection{u} \cap \wedgeselection{v}}$. 
    By the conclusion above, no vertex that does not attach to both~${\wedgeselection{u}}$ and~${\wedgeselection{v}}$ lies on a $(P,Q)$-path in~${G - (\barrier{u} \cup \barrier{v})}$. 
    Since~${\{u,v\}}$ is good, by \zcref{lem:attachingpair}, no vertex~${w \in I}$ that attaches to both~${\wedgeselection{u}}$ and~${\wedgeselection{v}}$ has a neighbour in both~${P}$ and~${Q}$. 
    Since~${I}$ is independent, we conclude that no ${(P,Q)}$-path in~${G - (\barrier{u} \cup \barrier{v})}$ can exist. 
\end{proof}

\begin{lemma}
    \label{lem:goodtripletopair}
    Let~$G$ be a cobweb with presentation~${(C,I)}$ and let~$\wedgeselectionnull$ be a wedge-selection for~${(G,C,I)}$. 
    Let~${u,v,w \in I}$ and let~$P$ be a component of~${\wedgeselection{u} \cap \wedgeselection{v} \cap \wedgeselection{w} \cap C}$. 
    Then there exists~$\{x,y\} \subseteq \{u,v,w\}$ such that~$P$ is a component of~${\wedgeselection{x} \cap \wedgeselection{y}}$. 
\end{lemma}

\begin{proof}
    Observe that~$P$ is a path. 
    Now each endpoint of~$P$ is an endpoint of at least one of the paths~${\wedgeselection{u} - u}$, ${\wedgeselection{v} - v}$, or~${\wedgeselection{w} - w}$ since otherwise, both neighbours of the endpoint of~$P$ on~$C$ would be contained in~${\wedgeselection{u} \cap \wedgeselection{v} \cap \wedgeselection{w} \cap C}$, contradicting that~$P$ is a component of~${\wedgeselection{u} \cap \wedgeselection{v} \cap \wedgeselection{w} \cap C}$. 
    So let~$\{x,y\} \subseteq \{u,v,w\}$ be any pair such that one endpoint of~$P$ is an endpoint of~${\wedgeselection{x} - x}$ and the other endpoint of~$P$ is an endpoint of~${\wedgeselection{y} - y}$. 
    Then, $P$ is a component of~${\wedgeselection{x} \cap \wedgeselection{y}}$, as desired.  
\end{proof}

We say a triple~$\{u,v,w\}$ of (not necessarily distinct) good vertices in~$I$ is a \emph{good triple}. 
So in particular, every~${\{x,y\} \subseteq \{u,v,w\}}$ is a good pair. 
As a direct consequence of this definition and \zcref{lem:goodtripletopair,lem:goodpair1}, we observe the following. 

\begin{corollary}
    \label{cor:goodtriple}
    Let~$G$ be a cobweb with presentation~${(C,I)}$ and let~$\wedgeselectionnull$ be a wedge-selection for~${(G,C,I)}$. 
    If~${\{u,v,w\} \subseteq I}$ is a good triple, then each component~$P$ of ${\wedgeselection{u} \cap \wedgeselection{v} \cap \wedgeselection{w} \cap C}$ is separated from~${C - V(P)}$ in~$G$ by the union of the three barriers~${\barrier{u} \cup \barrier{v} \cup \barrier{w}}$. \qed
\end{corollary}

Given a cobweb~$G$ with presentation~${(C,I)}$ and a wedge-selection~$\wedgeselectionnull$ for~${(G,C,I)}$, we denote by~$\mathcal{T}$ the set of non-empty sets of size at most~$3$ of good vertices. 
We call a function~$\segmentselectionnull$ mapping each~${\{u,v,w\} \in \mathcal{T}}$ to a component~$\segmentselection{\{u,v,w\}}$ of~${\wedgeselection{u} \cap \wedgeselection{v} \cap \wedgeselection{w} \cap C}$ for every~${\{u,v,w\} \in \mathcal{T}}$ a \emph{segment-selection for~$(C,I,\wedgeselectionnull)$}. 
Recall that a family of graphs is said to be \emph{intersecting} if the pairwise intersections of the vertex sets of these graphs are non-empty. 

\begin{lemma}
    \label{lem:separatinggoodpair}
    Let~$G$ be a cobweb with presentation~${(C,I)}$, let~$\wedgeselectionnull$ be a wedge-selection for~${(G,C,I)}$, and let~$\segmentselectionnull$ be a segment-selection for~$(C,I,\wedgeselectionnull)$ with image~$\mathcal{P}$. 
    If~$\mathcal{P}$ is intersecting, then there exists a good pair~${\{u,v\}}$ such that~${S(\{u,v\}) \coloneqq \barrier{u} \cup \barrier{v}}$ separates~$I$ from~${\segmentselection{\{u,v\}}}$ in~$G$. 
\end{lemma}

\begin{proof}
    Let~$\{u,v\}$ be a good pair such that~$\segmentselection{\{u,v\}}$ is $\subseteq$-minimal among all good pairs. 
    Note that we allow~${u = v}$ for this minimal choice. 
    
    Let~${S \coloneqq \barrier{u} \cup \barrier{v}}$. 
    To complete the proof, we need to show that $S$ separates~$I$ from~${\segmentselection{\{u,v\}}}$ in~$G$. 
    Consider the set~$I_0$ of vertices in ${I \setminus \{u,v\}}$ that have all their neighbours in~$\segmentselection{\{u,v\}}$. 
    By \zcref{lem:goodpair1}, each~${w \in I \setminus (I_0 \cup \{u,v\})}$ is separated by~$S$ from~${\segmentselection{\{u,v\}}}$. 
    Moreover, the neighbours of~$u$ and~$v$ that are not in~$S$ are in~${C - V(\segmentselection{\{u,v\}})}$, hence,~$u$ and~$v$ are separated from~$\segmentselection{\{u,v\}}$ by~$S$, again by \zcref{lem:goodpair1}. 
    So it suffices to prove that~${I_0}$ is empty. 
    
    \begin{figure}[htbp]
        \centering
        \includegraphics[scale=0.50]{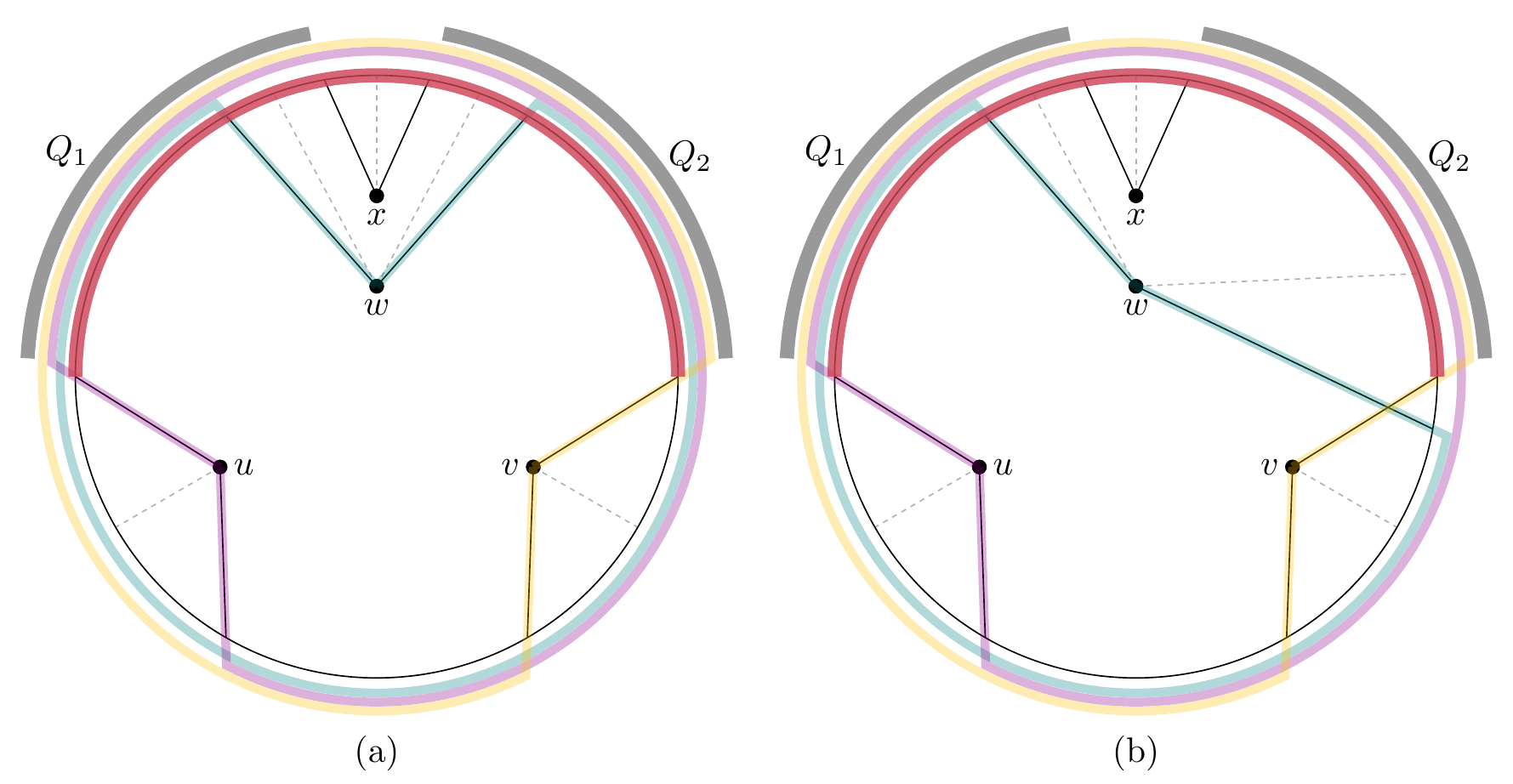}\\
        \vspace*{-0.3cm}
        \includegraphics[scale=0.50]{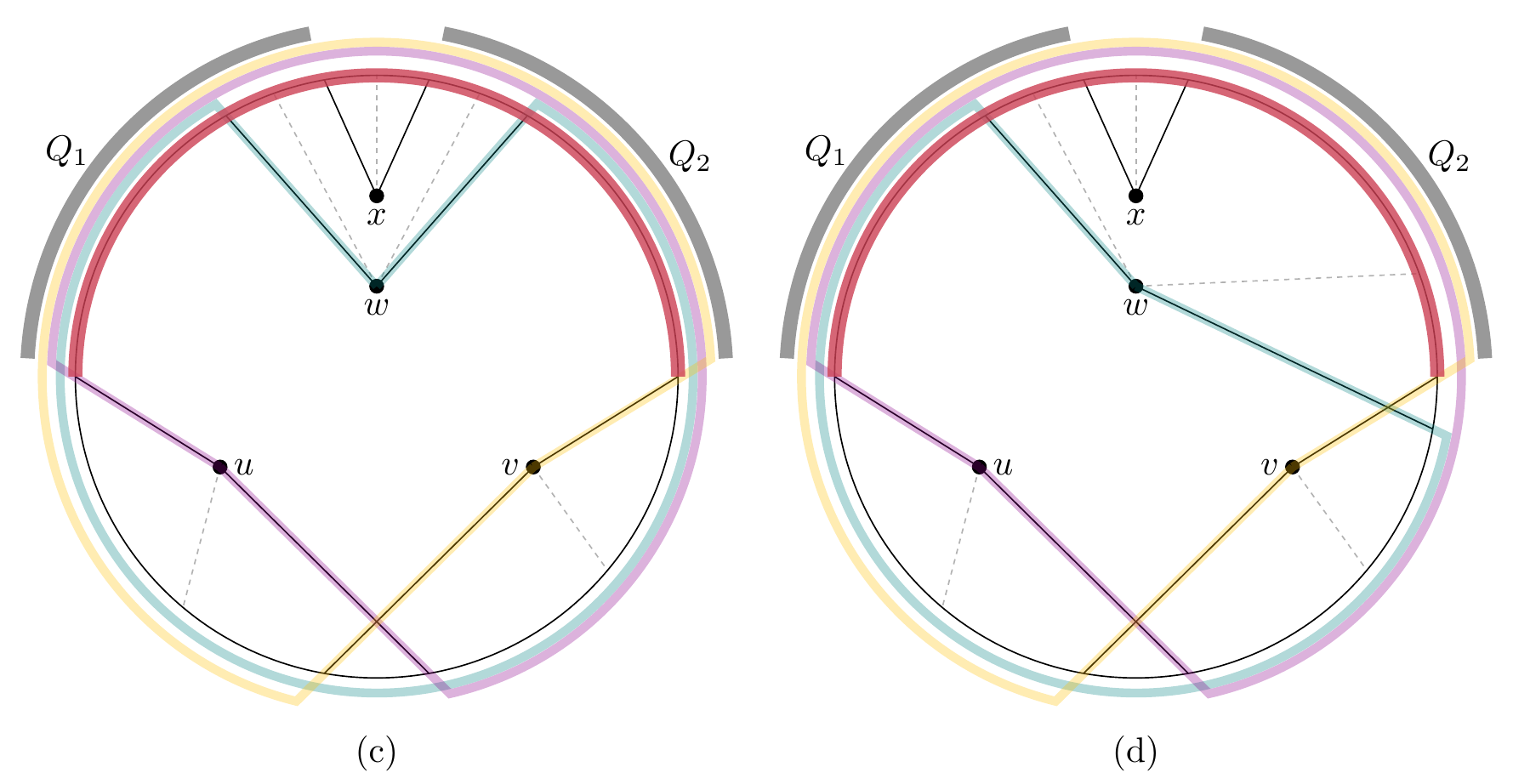}\\
        \vspace*{-0.3cm}
        \includegraphics[scale=0.5]{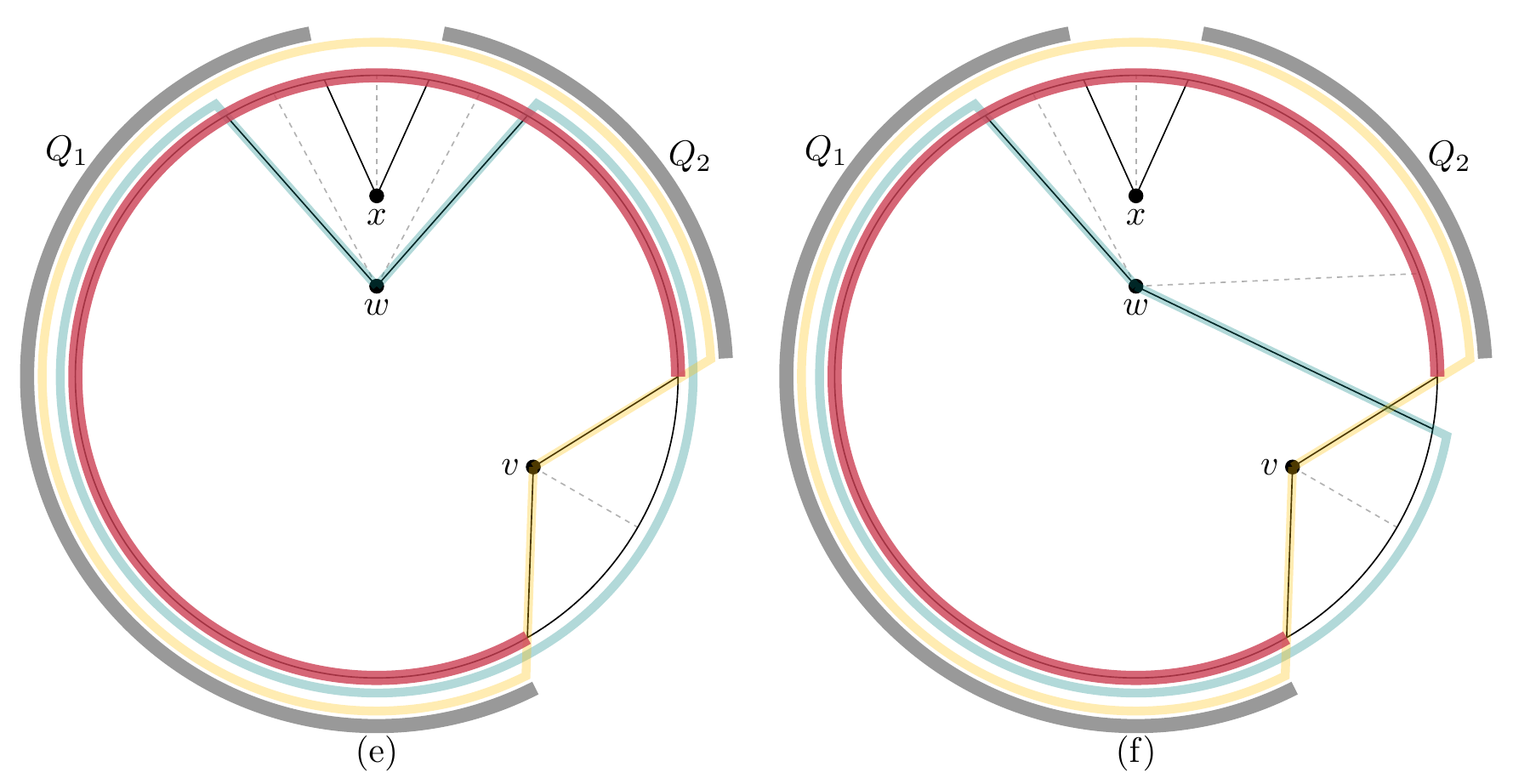}
        \caption{The situation in the proof of \zcref{lem:separatinggoodpair}, distinguished by the respective cases. 
            The red segment indicates~$\segmentselection{\{u,v\}}$. }
        \label{fig:CobwebsGoodPair}
    \end{figure}

    Let us assume for a contradiction that~$I_0$ is non-empty. 
    Let ${x \in I_0}$ be a vertex such that the path ${\wedgeselection{x} - x}$ is $\subseteq$-minimal among all vertices in~$I_0$. 
    Then~$x$ attaches to both~$\wedgeselection{u}$ and~$\wedgeselection{v}$, and since~$\wedgeselection{x}$ does neither improve~$\wedgeselection{u}$ nor~$\wedgeselection{v}$, both~$u$ and~$v$ attach to~$\wedgeselection{x}$ by \zcref{obs:AttachPairOrImprove}. 
    Moreover, we observe that~${\segmentselection{\{u,v\}} \cap \wedgeselection{x}}$ has two components; we denote the component containing an endpoint of~${\segmentselection{\{u\}}}$ by~$Q_1$ and the component containing an endpoint of~${\segmentselection{\{v\}}}$ by~$Q_2$ (if~${u = v}$, we allocate these names arbitrarily). 
    
    Let~${w \in I}$ be any good vertex with~${w \preceq x}$. 
    Then by definition, $\{u,v,w\}$ is a good triple. 
    Since~$\mathcal{P}$ is intersecting, and hence~${\segmentselection{\{w\}}}$ intersects~$\segmentselection{\{u,v\}}$, and since~${w \preceq x}$, 
    at least one endpoint of~${\segmentselection{\{w\}}}$ is contained in either~$Q_1$ or~$Q_2$. 
    Without loss of generality, let us say it is contained in~$Q_1$. 
    If the other endpoint of~$\segmentselection{\{w\}}$ would be contained in~${(\wedgeselection{u} \cap \wedgeselection{v}) \setminus \segmentselection{\{u,v\}}}$, then~${w \prec v}$, contradicting that~$v$ is good. 
    So the other endpoint of~${\segmentselection{\{w\}}}$ is contained in either~$Q_2$ (see \zcref{fig:CobwebsGoodPair}(a), (c), or (e)), or in~${C - V(\wedgeselection{v})}$  (see \zcref{fig:CobwebsGoodPair}(b), (d), or (f)). 
    But then, $\segmentselection{\{u,v,w\}}$ is a subset of either~$Q_1$ or~$Q_2$ since~$\mathcal{P}$ is intersecting. 
    By \zcref{lem:goodtripletopair}, we have that$~\segmentselection{\{u,v,w\}}$ is equal to either~$\segmentselection{\{u,w\}}$ or~$\segmentselection{\{v,w\}}$. 
    But both of these contradict the minimality of~$\segmentselection{\{u,v\}}$. 
    Hence, all cases lead to a contradiction to~${x \in I_0}$, so~$I_0$ is empty. 

    Therefore, for the good pair~$\{u,v\}$, we have that~$I$ is separated from~$\segmentselection{\{u,v\}}$ in~$G$ by~$S$, as desired. 
\end{proof}

\subsection{%
\texorpdfstring{Proof of \zcref{lem:NoBigComponents}}%
{Proof of Lemma 3.5}}
\label{sec:mainlemma}

We now use the tools developed in the previous subsection to prove \zcref{lem:NoBigComponents}. 
Before that, let us briefly discuss how we can deal with vertices that have nested neighbourhoods. 

\begin{lemma}
    \label{lem:identifySiblings}
    Let~$G$ be a graph, let~$\mathsf{w}$ be a normal weighting of~$G$, and let~${u, v \in V(G)}$ such that~${N(u) \subseteq N[v]}$. 
    Let~${G' \coloneqq G - u}$ and let~$\mathsf{w}'$ be the weighting of~$G'$ defined by 
    \[
        \mathsf{w}'(x) \coloneqq 
        \begin{cases}
            \mathsf{w}(u) + \mathsf{w}(v)  & \textnormal{ if } x = v\\
            \mathsf{w}(x)                   & \textnormal{ if } x \neq v \,.
        \end{cases}
    \]
    Then every $\mathsf{w}'$-balanced separator~$S$ of~$G'$ that does not contain~$v$ is a $\mathsf{w}$-balanced separator of~$G$. 
\end{lemma}

\begin{proof}
    Observe that $\mathsf{w}'$ is a normal weighting of~$G'$. 
    Let~$S$ be a $\mathsf{w}'$-balanced separator~$S$ of~$G'$ that does not contain~$v$. 
    Then $\mathsf{w}'(v) \leq \nicefrac{1}{2}$. 
    Since~${N(u) \subseteq N[v]}$, either~$u$ and~$v$ are in the same component of~${G-S}$, or~${N(u) \subseteq S}$. 
    If~${N(u) \subseteq S}$, then~$G[\{u\}]$ is a component of~${G-S}$ with~$\mathsf{w}(G[\{u\}]) = \mathsf{w}(u) \leq \mathsf{w}'(v) \leq \nicefrac{1}{2}$. 
    If~$D \neq G[\{u\}]$ is a component of~${G - S}$, 
    then~${\mathsf{w}(D) \leq \mathsf{w}'(D-u) \leq \nicefrac{1}{2}}$. 
    So~$S$ is a $\mathsf{w}$-balanced separator of~$G$. 
\end{proof}

We are now ready to prove \zcref{lem:NoBigComponents}, which we restate for the reader's convenience. 

\getkeytheorem{mainlemma}

\begin{proof}
    Let~$I$ be the set of all vertices in~$J$ of degree at least~$2$. 

    Let us first discuss how we can turn a separator that is not $\mathsf{w}$-balanced but its heavy component does not contain any vertices of~$I$ into a $\mathsf{w}$-balanced separator. 
    
    \begin{claim}
        \label{clm:buildseparator}
        Suppose a non-empty set~${X \subseteq V(C)}$ is not a $\mathsf{w}$-balanced separator of~$G$, but the unique heavy component~$B$ of~${G-X}$ does not contain any vertex in~$I$. 
        Then, there is a vertex~${w \in V(B \cap C)}$ such that~${X \cup \{w\}}$ is a $\mathsf{w}$-balanced separator of~$G$. 
    \end{claim}
    
    \begin{subproof}
        Since each vertex in~${J \setminus I}$ has degree~$1$ and~${B-J}$ is a path, we observe that~$B$ is a tree and hence has a $\mathsf{w}$-balanced separator~$\{w'\}$ by \zcref{obs:TreeBalancedSep}. 
        If~${w' \in V(C)}$, then we set~${w \coloneqq w'}$.
        If~${w' \notin V(C)}$, then~${w' \in J \setminus I}$, so~$w'$ has a unique neighbour~$w$, and we observe that~${w \in V(C)}$. 
        But since~${\mathsf{w}(w') \leq \nicefrac{1}{2}}$, we conclude that~$\{w\}$ is a balanced separator of~${B}$ as well. 
        Now it is easy to see that~${X \cup \{w\}}$ is a $\mathsf{w}$-balanced separator of~$G$. 
    \end{subproof}

    So if~$I$ is empty, we can take~$X$ to be any vertex of~$C$, and if~$X$ is not $\mathsf{w}$-balanced separator of~$G$, then the unique heavy component~$B$ of~${G-X}$ does not contain any vertex in~$I$. 
    Hence, by \zcref{clm:buildseparator}, we have a $\mathsf{w}$-balanced separator of~$G$ of size~$2$.
    So we may assume that~$I$ is non-empty. 
      
    \begin{claim}
        \label{clm:nosiblings}
        We may assume that for each pair of distinct~${u, v \in J}$, we have~${N(u) \not\subseteq N(v)}$. 
    \end{claim}

    \begin{subproof}
        We iteratively apply the construction of \zcref{lem:identifySiblings} until we have no such pair of vertices; let us call the resulting graph~$G^\ast$ and the resulting weighting~$\mathsf{w}^\ast$. 
        Since by \zcref{lem:identifySiblings}, any $\mathsf{w}^\ast$-balanced separator~${S \subseteq V(C)}$ of~$G^\ast$ is a $\mathsf{w}$-balanced separator~${S \subseteq V(C)}$ of~$G$, it suffices to show that ${\mathsf{w}^\ast(v) \leq \nicefrac{1}{2}}$ for all~${v \in V(G^\ast - V(C))}$. 
        Suppose that there is a vertex~${v_1 \in V(G^\ast - V(C))}$ with~$\mathsf{w}^\ast(v_1) > \nicefrac{1}{2}$. 
        Let~$v_2, \dots, v_m$ be the vertices in~$G$ such that~$N(v_i) \subseteq N(v_1)$ for each~${i \in [m]\setminus\{1\}}$ and~${\mathsf{w}^\ast(v_1) = \sum_{i = 1}^{m} \mathsf{w}(v_i)}$ obtained from the construction of~$G^\ast$. 
        Then consider~${S^\ast \coloneqq N(v_1)}$. 
        By \zcref{obs:WheelFromLargeNeighbourhood}, ${\lvert S^\ast \rvert \leq \ell-1}$. 
        But now, by construction, for each~${i \in [m]}$, we have that~$G[\{v_i\}]$ is a component of $G-S^\ast$, so each of them has $\mathsf{w}$-weight at most~$\nicefrac{1}{2}$. 
        Since~$\mathsf{w}$ is normal and~${\sum_{i = 1}^{m} \mathsf{w}(v_i) > \nicefrac{1}{2}}$, each other component has $\mathsf{w}$-weight at most~$\nicefrac{1}{2}$. 
        But then, $S^\ast$ is a $\mathsf{w}$-balanced separator of~$G$ such that ${S^\ast\subseteq V(C)}$. 
        Thus, $\mathsf{w}^\ast(v) \leq \nicefrac{1}{2}$ for every~${v \in V(G^\ast)}$, and therefore, we may assume without loss of generality, that~${G = G^\ast}$ and~${\mathsf{w} = \mathsf{w}^\ast}$. 
    \end{subproof}
    
    We conclude that~${G' \coloneqq G[C \cup I]}$ is a cobweb presented by~$(C,I)$, using~\zcref{clm:nosiblings} and the assumption that~$I$ is non-empty. 

    As a direct consequence of \zcref{obs:WheelFromLargeNeighbourhood}, each~${v \in I}$ has degree less than~$\ell$. 
    Moreover, consider a wedge~$W$ of~${(G,C,I)}$. 
    If~$W$ is trivial or co-trivial, then, by definition the barrier~$\barrier{W}$ contains~$2$ vertices. 
    Otherwise, the barrier~$\barrier{W}$ contains fewer than~$\ell$ vertices, since contracting the component of $G-V(W)$ containing the vertices of $C-V(W)$ would result in a vertex with $\barrier{W}$ as the set of neighbours in~$W$, implying that~$G$ contains~$W_{\abs{\barrier{W}}}$ as an induced minor. 
    Now, for each~${v \in I}$ consider all wedges of~${(G,C,I)}$ anchored in~$v$. 
    Since $v$ has degree less than~$\ell$, there are at most~$\ell-1$ wedges anchored in~$v$. 
    If for some vertex~${v \in I}$ the union 
    \[
        {S(v) \coloneqq \bigcup \{ \barrier{W} \colon W \textnormal{ is a wedge anchored in } v \}}
    \]
    is a $\mathsf{w}$-balanced separator of~$G$, there is nothing else to show since~${S(v) \subseteq V(C)}$ and ${\abs{S(v)} \leq (\ell-1)^2}$. 
    So assume that for each~${v \in I}$ there is a heavy component~$B(v)$ of~${G - S(v)}$. 
    Note that by \zcref{obs:CobwebAttaches}, for any vertex~${v \in I}$ and any wedge~$W$ anchored in~$v$, any vertex~${u \in I}$ that does not attach to $W$ is separated from~$W$ by the barrier of~$W$. 
    So there is a wedge~$\wedgeselection{v}$ anchored in~$v$ such that~${V(B(v))}$ is contained in the union of~${V(\wedgeselection{v}-v)}$ 
    that have all their neighbours in~${\wedgeselection{v}-v}$. 
    Moreover, we observe that for each~${v \in {I}}$, 
    \begin{equation*}\label{eq:component}
        \tag{$\ast$}
        B(v) \textnormal{ is also a component of }{G - \barrier{v}}. 
    \end{equation*}
    Now let~$\wedgeselectionnull$ with ${v \mapsto \wedgeselection{v}}$ for all~${v \in I}$ be the resulting wedge-selection for~${(G,C,I)}$. 
    
    Let~$\mathcal{T}$ denote the set of good triples with respect to~$\wedgeselectionnull$ as defined in the previous section. 
    If for some good triple~${\{u,v,w\} \in \mathcal{T}}$ the set ${S(\{u,v,w\}) \coloneqq \barrier{u} \cup \barrier{v} \cup \barrier{w}}$ is a balanced separator of~$G$, then there is nothing else to show since~${S(\{u,v,w\}) \subseteq V(C)}$ and ${\abs{S(\{u,v,w\})} \leq 3(\ell-1) \leq (\ell-1)^2}$. 
    So we may assume that for each~${\{u,v,w\} \in \mathcal{T}}$ there exists a heavy component~$B(\{u,v,w\})$ of~${G-S(\{u,v,w\})}$. 
    Since~$B(\{u,v,w\})$ is a component of ${G-S(\{u,v,w\})}$, it is contained in a component of ${G - \barrier{u}}$, namely in $B(u)$, as ${\mathsf{w}(B(\{u,v,w\})) > \nicefrac{1}{2}}$ and $B(u)$ is the unique component of ${G - \barrier{u}}$ with weight more than $\nicefrac{1}{2}$. 
    Similarly, ${B(\{u,v,w\}) \subseteq B(v)}$ and ${B(\{u,v,w\}) \subseteq B(w)}$, hence, ${B(\{u,v,w\}) \subseteq B(u) \cap B(v) \cap B(w)}$. 
    By our choice of~$\wedgeselectionnull$, ${B(x) \cap C \subseteq \wedgeselection{x}}$ for all~${x \in I}$, so~${B(u) \cap B(v) \cap B(w) \cap C \subseteq \wedgeselection{u} \cap \wedgeselection{v} \cap \wedgeselection{w} \cap C}$. 
    Note that~${B(\{u,v,w\}) \cap C}$ is non-empty since~$S(\{u,v,w\})$ is not a $\mathsf{w}$-balanced separator. 
    Since~${B(\{u,v,w\})}$ is a component of~$G-S(\{u,v,w\})$, by \zcref{cor:goodtriple}, it does not intersect multiple components of ${\wedgeselection{u} \cap \wedgeselection{v} \cap \wedgeselection{w} \cap C}$, so ${B(\{u,v,w\}) \cap C}$ is a subgraph of a unique component of~${\wedgeselection{u} \cap \wedgeselection{v} \cap \wedgeselection{w} \cap C}$; we denote that component by~$\segmentselection{\{u,v,w\}}$. 
    Finally, let~$\segmentselectionnull$ with~${\{u,v,w\} \mapsto \segmentselection{\{u,v,w\}}}$ for all ${\{u,v,w\} \in \mathcal{T}}$ be the resulting segment-selection for~${(C,I,\wedgeselectionnull)}$. 

    \begin{claim}
        \label{clm:mainclaim}
        There exists a good pair~${\{u,v\} \in \mathcal{T}}$ such that~${V(B(\{u,v\})) \cap I}$ is empty. 
    \end{claim}
    
    \begin{subproof}
        To be able to use \zcref{lem:separatinggoodpair}, we need to prove that 
        the image of~$\segmentselectionnull$ is intersecting. 
        By our choice for the wedge-selection~$\wedgeselectionnull$, for each~${w \in I}$ we have that~$V(B(w))$ is contained in the union of~${V(\wedgeselection{w}-w) = V(\segmentselection{\{w\}})}$ with the set of vertices in~$J$ that have all their neighbours in~${\segmentselection{\{w\}}}$. 
        Similarly, by our choice for the segment-selection~$\segmentselectionnull$, we can observe that for every good triple~${T \in \mathcal{T}}$, the vertex set of~$B(T)$ is contained in the union of~$V(\segmentselection{T})$ with the set of vertices in~$J$ that have all their neighbours in~$\segmentselection{T}$. 
        So, assume for a contradiction that there exist two distinct~${T, T' \in \mathcal{T}}$ such that~$\segmentselection{T}$ and~$\segmentselection{T'}$ are disjoint. 
        Then, the respective~$B(T)$ and~$B(T')$ are disjoint as well, and hence~${\mathsf{w}(G) \geq \mathsf{w}(B(T) \cup B(T')) = \mathsf{w}(B(T)) + \mathsf{w}(B(T')) > 1}$, a contradiction to the normality of~$\mathsf{w}$.  
        Therefore, 
        the image of~$\segmentselectionnull$ is intersecting and, by \zcref{lem:separatinggoodpair}, there exists a good pair~${\{u,v\}}$ such that~${S(\{u,v\})}$ separates~$I$ from ${\segmentselection{\{u,v\}}}$ in~$G'$, and hence in~$G$. 
        Therefore, ${V(B(\{u,v\})) \subseteq V(\segmentselection{\{u,v\}}) \cup (J \setminus I)}$. 
    \end{subproof}

    By \zcref{clm:mainclaim}, there exists a good pair~${\{u,v\}}$ such that ${V(B(\{u,v\})) \cap I}$ is empty.
    Applying \zcref{clm:buildseparator} to the set~$S(\{u,v\})$ shows that there exists a vertex~${w \in V(B(\{u,v\}) \cap C)}$ such that the set~$S\coloneqq S(\{u,v\}) \cup \{w\}$ is a $\mathsf{w}$-balanced separator of~$G$. 
    Since~${S \subseteq V(C)}$ and~${\abs{S} \leq {2(\ell-1) + 1 \leq (\ell-1)^2}}$, this concludes the proof of the lemma. 
\end{proof}

\section*{Acknowledgements}

The authors are grateful to Ekkehard K\"{o}hler and Pawe\l{} Rzążewski for helpful discussions. 
This work was conducted in part when the first author visited University of Primorska. 

\printbibliography

\appendix

\section{Dominated balanced separators in fan-induced-minor-free graphs}

Given an integer~${\ell \geq 2}$, let~$F_\ell$ denote the \emph{$\ell$-fan}, that is, the graph obtained from an $\ell$-vertex path by adding a universal vertex. 
Clearly, $F_{\ell}$ is an induced subgraph of~$W_{\ell+1}$, so \zcref{thm:mainWeighted} immediately implies that the class of~$F_\ell$-induced-minor-free graphs has balanced separators dominated by few vertices. 
However, this result is a much more straightforward consequence of \zcref{lem:modifyBalancedSep,lem:GyarfasPath}, which we prove here. 

\begin{theorem}
    \label{thm:fanfree}
    For every integer~${\ell \geq 2}$ and every $F_\ell$-induced minor free graph~$G$ with weighting~$\mathsf{w}$ of~$G$ there is a $\mathsf{w}$-balanced separator of~$G$ that is dominated by at most~$\ell$ vertices. 
\end{theorem}

\begin{proof}
    Let~${P = p_1 p_2 \dots p_k}$ be a minimal path as in \zcref{lem:GyarfasPath}.
    Since~${N[P]}$ is a $\mathsf{w}$-balanced separator of~$G$ by \zcref{lem:GyarfasPath}, we may assume~${k \geq 3}$. 
    Let~${Q \coloneqq p_1 \dots p_{k-1}}$. 
    By the minimality of~$P$ there is a unique heavy component~$B$ of~${G - V(Q)}$. 
    Applying \zcref{lem:modifyBalancedSep} with~${X \coloneqq N[Q]}$,~${Y \coloneqq N[p_k]}$, and~${Z \coloneqq V(Q)}$, we conclude that the set ${S \coloneqq N[N(B) \cup \{p_k\}]}$ is a $\mathsf{w}$-balanced separator of~$G$. 
    Note that the graph obtained from~$G[Q \cup B]$ by contracting~$B$ into a single vertex and successively deleting \hbox{degree-$1$} vertices and suppressing degree-$2$ vertices is a $\abs{N(B)}$-fan as long as~${\abs{N(B)} \geq 2}$, so~${\abs{N(B)} < \ell}$ since~$G$ is $F_\ell$-induced minor free. 
    Hence, $S$ is dominated by at most~$\ell$ vertices, as required. 
\end{proof}

\end{document}